\documentclass[a4paper,12pt]{amsart}

\usepackage[T1]{fontenc}
\usepackage[utf8]{inputenc}

\usepackage{amsmath}
\usepackage{amsthm}
\usepackage{amssymb}
\usepackage{mathtools}
\usepackage{enumitem}
\usepackage{nicefrac}

\usepackage{mathrsfs}

\usepackage[authoryear]{natbib}
\usepackage[margin=3.25cm]{geometry}

\usepackage[font={footnotesize},format=hang]{caption}

\usepackage{hyperref}

\usepackage{xcolor}
\hypersetup{
    colorlinks,
    linkcolor={red!90!black},
    citecolor={green!70!black},
    urlcolor={blue!80!black}
}

\definecolor{myteal}{HTML}{00B7AF}
\definecolor{mygold}{HTML}{FDD401}
\definecolor{mycoral}{HTML}{FD5A19}

\usepackage{tikz}
\usetikzlibrary{snakes}

{\noindent\color{magenta}{RG: \  }}{\hfill{$\square$}\color{black}\par\medskip}

{\noindent\color{blue}{ZZ: \  }}{\hfill{$\square$}\color{black}\par\medskip}

\newtheorem{theorem}{Theorem}[section]
\newtheorem{lemma}[theorem]{Lemma}
\newtheorem{corollary}[theorem]{Corollary}

\newtheoremstyle{mytheoremstyle} 
    {1em}                    
    {\topsep}                    
    {\rmfamily}                   
    {}                           
    {\bfseries}                   
    {.}                          
    {.5em}                       
    {}  

\theoremstyle{mytheoremstyle}

\newtheorem{definition}[theorem]{Definition}

\newtheorem{remark}[theorem]{Remark}

\newcommand\defeq{\coloneqq} 
\newcommand\dftt{\mathtt{df}\,}

\newcommand\Ratio{\mathbf{Q}} 
\newcommand\Real{\varmathbb{R}} 

\newcommand\Rat{\varmathbb{Q}}

\DeclareMathOperator{\Ult}{Ult} 
\newcommand{\ult}{\mathscr{U}} 

\newcommand\Bmod{\mathop{\langle B\rangle}} 
\newcommand\Nec{\mathop{\left[B\right]}} 
\DeclareMathOperator{\Conv}{C} 
\DeclareMathOperator{\at}{@}

\newcommand{\tand}{\text{ and }}
\newcommand{\tor}{\text{ or }}
\newcommand{\qtor}{\quad\text{ or }\quad}
\newcommand{\qtand}{\quad\text{ and }\quad}

\DeclareMathOperator{\Ex}{\mathsf{E}}
\DeclareMathOperator{\All}{\mathsf{A}}
\DeclareMathOperator{\ue}{\mathfrak{ue}} 

\newcommand{\Rue}{\mathrel{R^{\ue}}}

\newcommand{\Var}{\mathrm{Prop}} 
\newcommand{\Nom}{\Omega} 
\DeclareMathOperator{\Th}{Th} 

\newcommand{\LBWE}{\boldsymbol{\mathsf{LBWE}}} 
\newcommand{\DLBWE}{\boldsymbol{\mathsf{DLBWE}}} 
\newcommand{\CDLBWE}{\boldsymbol{\mathsf{CDLBWE}}} 
\newcommand{\SCDLBWE}{\boldsymbol{\mathsf{SCDLBWE}}} 
\newcommand{\DLOWE}{\boldsymbol{\mathsf{DLOWE}}} 
\newcommand{\LOWE}{\boldsymbol{\mathsf{LOWE}}} 

\newcommand{\Klass}{\boldsymbol{\mathsf{K}}} 
\newcommand{\bDelta}{\boldsymbol{\Phi}}

\newcommand{\gsub}{\rightarrowtail} 
\newcommand\Iffdef{\;\mathrel{\mathord{:}\mathord{\Longleftrightarrow}}\;}
\newcommand\iffdef{\;\mathrel{\mathord{:}\mathord{\longleftrightarrow}}\;}

\newcommand{\Bh}{\ensuremath{\mathbf{B}_h}}
\newcommand{\Bhp}{\ensuremath{\mathbf{B}_h^+}}
\newcommand{\frFG}{\frF^{\Gamma}}
\newcommand{\frFGp}{\frFG_{+}}
\newcommand{\frMG}{\frM^{\Gamma}}
\newcommand{\frMGp}{\frMG_{+}}
\newcommand{\frGast}{\frG^{\ast}}
\newcommand{\frMast}{\frM^{\ast}}
\newcommand{\VG}{V^{\Gamma}}
\newcommand{\VGp}{\VG_{+}}

\newcommand{\BGp}{B^{\Gamma}_{+}}

\newcommand{\Lb}{\calL_{\beta}} 
\newcommand{\Hb}{\calH_{\beta}} 
\newcommand{\Hbat}{\calH_{\beta}(\at)} 
\newcommand{\Hbex}{\calH_{\beta}(\Ex)} 
\newcommand{\Hbatex}{\calH_{\beta}(\at,\Ex)} 

\newcommand{\Iff}{\Longleftrightarrow}
\newcommand\Rarrow{\Rightarrow}
\newcommand\Larrow{\Leftarrow}
\newcommand\rarrow{\rightarrow}
\renewcommand\iff{\longleftrightarrow}

\newcommand{\cal}{\mathcal}
\newcommand{\calH}{\cal{H}}
\newcommand{\calL}{\cal{L}}

\newcommand\mfr{\mathfrak}
\newcommand\frF{\mfr{F}}
\newcommand\frG{\mfr{G}}
\newcommand\frI{\mfr{I}}
\newcommand\frL{\mfr{L}}
\newcommand\frM{\mfr{M}}
\newcommand\frQ{\mfr{Q}}
\newcommand\frR{\mfr{R}}
\newcommand\frW{\mfr{W}}

\newcommand\suchthat{\,\middle\vert\,} 

\makeatletter
\DeclareSymbolFont{lettersA}{U}{txmia}{m}{it}
\SetSymbolFont{lettersA}{bold}{U}{txmia}{bx}{it}
\DeclareFontSubstitution{U}{txmia}{m}{it}

\DeclareMathSymbol{\m@thbbch@rA}{\mathord}{lettersA}{129}
\DeclareMathSymbol{\m@thbbch@rB}{\mathord}{lettersA}{130}
\DeclareMathSymbol{\m@thbbch@rC}{\mathord}{lettersA}{131}
\DeclareMathSymbol{\m@thbbch@rD}{\mathord}{lettersA}{132}
\DeclareMathSymbol{\m@thbbch@rE}{\mathord}{lettersA}{133}
\DeclareMathSymbol{\m@thbbch@rF}{\mathord}{lettersA}{134}
\DeclareMathSymbol{\m@thbbch@rG}{\mathord}{lettersA}{135}
\DeclareMathSymbol{\m@thbbch@rH}{\mathord}{lettersA}{136}
\DeclareMathSymbol{\m@thbbch@rI}{\mathord}{lettersA}{137}
\DeclareMathSymbol{\m@thbbch@rJ}{\mathord}{lettersA}{138}
\DeclareMathSymbol{\m@thbbch@rK}{\mathord}{lettersA}{139}
\DeclareMathSymbol{\m@thbbch@rL}{\mathord}{lettersA}{140}
\DeclareMathSymbol{\m@thbbch@rM}{\mathord}{lettersA}{141}
\DeclareMathSymbol{\m@thbbch@rN}{\mathord}{lettersA}{142}
\DeclareMathSymbol{\m@thbbch@rO}{\mathord}{lettersA}{143}
\DeclareMathSymbol{\m@thbbch@rP}{\mathord}{lettersA}{144}
\DeclareMathSymbol{\m@thbbch@rQ}{\mathord}{lettersA}{145}
\DeclareMathSymbol{\m@thbbch@rR}{\mathord}{lettersA}{146}
\DeclareMathSymbol{\m@thbbch@rS}{\mathord}{lettersA}{147}
\DeclareMathSymbol{\m@thbbch@rT}{\mathord}{lettersA}{148}
\DeclareMathSymbol{\m@thbbch@rU}{\mathord}{lettersA}{149}
\DeclareMathSymbol{\m@thbbch@rV}{\mathord}{lettersA}{150}
\DeclareMathSymbol{\m@thbbch@rW}{\mathord}{lettersA}{151}
\DeclareMathSymbol{\m@thbbch@rX}{\mathord}{lettersA}{152}
\DeclareMathSymbol{\m@thbbch@rY}{\mathord}{lettersA}{153}
\DeclareMathSymbol{\m@thbbch@rZ}{\mathord}{lettersA}{154}

\long\def\DoLongFutureLet #1#2#3#4{%
   \def\@FutureLetDecide{#1#2\@FutureLetToken
      \def\@FutureLetNext{#3}\else
      \def\@FutureLetNext{#4}\fi\@FutureLetNext}
   \futurelet\@FutureLetToken\@FutureLetDecide}
\def\DoFutureLet #1#2#3#4{\DoLongFutureLet{#1}{#2}{#3}{#4}}
\def\@EachCharacter{\DoFutureLet{\ifx}{\@EndEachCharacter}%
   {\@EachCharacterDone}{\@PickUpTheCharacter}}
\def\m@keCharacter#1{\csname\F@ntPrefix#1\endcsname}
\def\@PickUpTheCharacter#1{\m@keCharacter{#1}\@EachCharacter}
\def\@EachCharacterDone \@EndEachCharacter{}

\DeclareRobustCommand*{\varmathbb}[1]{\gdef\F@ntPrefix{m@thbbch@r}%
  \@EachCharacter #1\@EndEachCharacter}
\makeatother

\allowdisplaybreaks

\date{}

\title{Hybrid logic of strict betweenness}
\author{Rafa\l{} Gruszczy\'nski, Zhiguang Zhao}

\address{Rafa\l\ Gruszczy\'nski\\
Department of Logic\\
Nicolaus Copernicus University in Toru\'n\\
Poland\\
\textsc{Orcid:} 0000-0002-3379-0577}
\email{gruszka@umk.pl}

\urladdr{www.umk.pl/\textasciitilde gruszka}

\address{Zhiguang Zhao\\
School of Mathematics and Statistics\\
Taishan University\\
P.R.China\\
\textsc{Orcid:} 0000-0001-5637-945X}
\email{zhaozhiguang23@gmail.com}

\begin{document}

\begin{abstract}
    The paper is devoted to modal properties of the ternary strict betweenness relation as used in the development of various systems of geometry. We show that such a~relation is non-definable in a~basic similarity type with a~binary operator of possibility, and we put forward two systems of hybrid logic, one of them complete with respect to the class of dense linear betweenness frames without endpoints, and the other with respect to its subclass composed of Dedekind complete frames.

    \medskip


\medskip

\end{abstract}

\maketitle

\section{Introduction}

In this paper we investigate mutual dependencies between a ternary strict betweenness relation and hybrid logic with a~binary operator of possibility and a~family of satisfaction operators in the sense of \citep[7.3 Hybrid logic]{Blackburn-et-al-ML}. The betweenness relation we put into focus is the cornerstone of order fragments of axiomatic systems of Hilbert-style geometry, and in choosing the constraints, we stay as close to geometry as possible. In this, our work differs from that of \cite{Duntsch-et-al-BA} where the authors scrutinize a class of reflexive betweenness frames that encompass very general aspects of betweenness, common to many, not necessarily geometric in spirit, approaches.  The other difference is hidden in the techniques applied: \cite{Duntsch-et-al-BA} use Boolean algebras with operators, while we adhere to logic. The two papers share the idea for the possibility operator $\Bmod$ as defined in \citep{vanBenthem-MLS}.

In Section~\ref{sec:betweenness}, we list the first-order axioms for strict betweenness, the inspiration for these being the axiomatization of betweenness by \cite{Borsuk-Szmielew-FG}. We single out two classes of frames: linear betweenness frames without endpoints ($\LBWE$) and dense linear betweenness frames without endpoints ($\DLBWE$). These will play a crucial role in the sequel. Moreover, in this section we draw a comparison between betweenness frames and standard binary linearly ordered structures. Section~\ref{sec:hybrid-logic} contains a semantical characterization of the modal language we use.

In Section~\ref{sec:non-definability} we deal with the problem of non-definability of first-order properties of betweennes either by means of a~pure modal language or a modal language enriched with the satisfaction operators. To do this, we generalize ten Cate's Goldblatt-Thomason style theorem \cite[Theorem 4.3.1]{tenCate-MTFEML} to arbitrary modal types. Applying the generalized theorem, we show that the class of dense linear betweenness frames without endpoints is not definable in the extended language with the satisfaction operators. In Section~\ref{sec:definability} we show that the class $\DLBWE$ can be defined by means of pure formulas of hybrid logic. Section~\ref{sec:definability-density} pursues the problem of definability, focusing on the particular case of the density axiom. We prove that the property can be expressed by means of a pure formula of the extended language with the satisfaction operators with respect to the class $\LBWE$. This is a key result that allows us to find a pure hybrid axiomatization of dense linear betweenness frames without endpoints in Section~\ref{sec:hybrid-logic-of-strict-betweenness}. Therein, we put forward logic $\Bh$ which is sound and complete with respect to $\DLBWE$. In particular, we obtain that $\Bh$ is a~logic of the frame $\frQ\defeq\langle\Rat,B_{<}\rangle$ obtained from the unique---up to isomorphism---countable dense linear order without endpoints, i.e.,  $\langle\Rat,<\rangle$.

In the same section, we go beyond the elementary classes of frames considering \emph{Dedekind complete} elements of $\DLBWE$, the class $\CDLBWE$. We propose a~modal counterpart of the completeness axiom couched in the language with the existential modality, and we consider the logic $\Bhp$ extending $\Bh$ with the axiom. We prove soundness and completeness of this logic with respect to $\CDLBWE$. In consequence, we prove that $\Bhp$ is a logic of the frame $\frR\defeq\langle\Real,B_{<}\rangle$ obtained from the real line $\langle\Real,<\rangle$.

Throughout, we use the standard logical notations with $\bot$, $\neg$, $\wedge$, $\vee$, $\rightarrow$, $\leftrightarrow$, $\forall$ and $\exists$ being, respectively, falsum constant, negation, conjunction, disjunction, material implication, material equivalence, universal and existential quantifier. The symbols $\Rarrow$ and $\Iff$ are reserved for meta-implication and meta-equivalence, $\Iffdef$ is an equivalence by definition. Concerning the terminology, unless otherwise stated, we follow \citep{Blackburn-et-al-ML}. In particular,  `frame' means always a Kripke frame, and an $n$-frame is a frame with a single $n$-ary relation.

\section{Betweenness relations and linear orders}\label{sec:betweenness}
Let $\frF\defeq\langle W,B\rangle$ be a 3-frame. We will read $B(x,y,z)$ as \emph{$y$ is between $x$ and $z$}, and in this way---for heuristic reasons---we distinguish the middle coordinate. $\#(x,y,z)$ means that the elements of the set $\{x,y,z\}$ are pairwise different.

For the basic characterization of (strict) betweenness, we are going to use the following standard set of first-order axioms:
\begin{gather}
    B(x,y,z)\rarrow\#(x,y,z)\,,\tag{B1}\label{B1}\\
    B(x,y,z)\rarrow B(z,y,x)\,,\tag{B2}\label{B2}\\
    B(x,y,z)\rarrow\neg B(x,z,y)\,,\tag{B3}\label{B3}\\
    B(x,y,z)\wedge B(y,z,u)\rarrow B(x,y,u)\,,\tag{B4}\label{B4}\\
    B(x,y,z)\wedge B(y,u,z)\rarrow B(x,y,u)\,,\tag{B5}\label{B5}\\
    \#(x,y,z)\rarrow B(x,y,z)\vee B(x,z,y)\vee B(y,x,z)\,,\tag{B6}\label{B6}\\
    \forall y\exists x\exists       z\,B(x,y,z)\,,\tag{B7}\label{B7}\\
    x\neq z\rarrow \exists y\,B(x,y,z)\,.\tag{B8}\label{B8}
\end{gather}

\eqref{B1} says that $B$ is a strict relation, which according to \eqref{B2} is symmetrical on its first and third projections. \eqref{B3} postulates asymmetry on second and third coordinates. \eqref{B4} and \eqref{B5} are, respectively, outer and inner transitivity axioms. \eqref{B6} stipulates linearity\footnote{Of course, in presence of \eqref{B2}. Without the symmetry axiom to stipulate linearity we would need six disjuncts in the consequent of the axiom. See also the proof of Theorem~\ref{th:non-definability} for the case of \eqref{B6}.}, \eqref{B7} is unboundedness axiom, and \eqref{B8} the density axiom. These are precisely the \emph{order} axioms that form a fragment of the system of geometry from \citep{Borsuk-Szmielew-FG}. The proofs of other elementary properties of betweenness that we use throughout the paper can be found in the appendix.

\begin{definition}
  Any 3-frame $\frF$ that satisfies \eqref{B1}--\eqref{B7} will be called \emph{a linear betweenness frame without endpoints}.  The class of all such frames will be denoted by $\LBWE$. The elements of the class
  \[
    \DLBWE\defeq\LBWE+\eqref{B8}
  \]
  will be called \emph{dense} linear betweenness frames without endpoints.
\end{definition}

Let $\DLOWE$ be the class of all dense linear orders without endpoints, i.e., the class of all 2-frames $\langle W,<\rangle$ satisfying the following axioms:
\begin{gather}
    x<y\rarrow y\not< x\,,\tag{L1}\\
    x<y\wedge y<z\rarrow x<z\,,\tag{L2}\\
    x\neq y\rarrow x<y\vee y<x\,,\tag{L3}\\
    \forall x\exists y\exists z\,y<x<z\,,\tag{L4}\\
    \forall x\forall y\,(x<y\rarrow\exists z(x<z<y))\,.\tag{L5}
\end{gather}
Let $\LOWE$ be the class of frames that satisfy first four constraints. It is obvious that if $\frL\defeq\langle W,<\rangle\in\DLOWE$, then the relation $B_<$ such that:
\[
B_<(x,y,z)\Iffdef x<y<z\tor z<y<x
\]
is a betweenness relation and $\langle W,B_<\rangle\in\DLBWE$ (analogous relation holds between $\LOWE$ and $\LBWE$). A slightly less obvious---albeit well-known---fact is that given a betweenness $B$ we may define a binary relation $<_B$ that meets the axioms for the class $\DLOWE$ (or $\LOWE$). Depending on the orientation we chose for the transition from $B$ to $<_{B}$, we obtain a pair of dually equivalent linear orders.\footnote{The details of the construction can be found in \citep{Borsuk-Szmielew-FG}.} Thus, there is a one-to-one correspondence between the class $\DLBWE$ (resp. $\LBWE$) and the class of pairs of dual orders from $\DLOWE$ (resp. $\LOWE$). This correspondence extends to non-elementary classes with the additional second-order completeness axiom (which we address in Section~\ref{sec:hybrid-logic-of-strict-betweenness}).

As it can be seen from the above, the notion of \emph{betweenness} of ours is a very particular one and bounded with geometry. Let us emphasize one more time that this is not the only approach to betweenness that can be found in the literature. For example, betweenness as obtained from binary relations is the subject matter of \citep{Altwegg-ZADTGM}, \citep{Sholander-TLOAB}, \citep{Dvelmeyer-et-al-ACOOSALVBR}, \citep{Duntsch-et-al-BACOFBR}, and also---as a strict relation---of \citep{Lihova-SOB} and \citep{Courcelle-BOPO}. A betweenness as a~ternary relation on vertices of a graph is studied in  \citep{Changat-et-al-BIGASSOSAIPB}, and as a concept on trees in \citep{Sholander-TLOAB} and \citep{Courcelle-BIOTT}. The works we cite form only a fraction of the literature devoted to betweenness, with the purpose being to turn the reader's attention to the fact that the relation can be studied from different perspectives.

\section{Hybrid logic with binary operators}\label{sec:hybrid-logic}

As the betweenness relation is ternary for its modal analysis we will need a binary modal operator $\Bmod$. Let $\beta$ be the modal similarity type with $\Bmod$ as the only operator. The modal language $\Lb$ is given by the following definition
\[
\varphi\defeq \top\mid p\mid \neg\varphi\mid\varphi\wedge\psi\mid\Bmod(\varphi,\psi)\,.
\]
The basic hybrid language $\Hb$ is a two-sorted language that is an expansion of $\Lb$ with nominals $i,j,k,l$ (indexed if necessary)
\[
\varphi\defeq \top\mid p\mid i\mid \neg\varphi\mid\varphi\wedge\psi\mid\Bmod(\varphi,\psi)\,.
\]
The set of all propositional letters will be denoted by `$\Var$', and the set of nominals by `$\Nom$'. We assume that $\Var\cap\Nom=\emptyset$.

Last but not least, the hybrid language that plays the most important role in this paper is $\Hbat$ that expands $\Hb$ with a family of satisfaction operators $\at_i$, one for every nominal
\[
\varphi\defeq \top\mid p\mid i\mid \neg\varphi\mid\varphi\wedge\psi\mid\Bmod(\varphi,\psi)\mid\at_i\varphi\,.
\]

The basic idea for an operator $\Bmod$ comes from \cite{vanBenthem-MLS}. Given a model $\frM\defeq\langle\frF,V\rangle$ based on a 3-frame $\frF$ we characterize the semantic for $\Bmod$ in the following way
\begin{equation}\tag{$\dftt{\Bmod}$}\label{df:Bmod}
\begin{split}
    \frM,w\Vdash\Bmod(\varphi,\psi)\Iffdef{}&\\ (\exists x,y\in W)\,&{}(\frM,x\Vdash\varphi\tand\frM,y\Vdash\psi\tand B(x,w,y))\,.
\end{split}
\end{equation}
$\Bmod$ gives rise to a natural unary \emph{convexity} operator
\begin{equation}\tag{$\dftt{\Conv}$}\label{df:Conv}
    \Conv\varphi\Iffdef\Bmod(\varphi,\varphi)\,.
\end{equation}

\begin{definition}
    Given a frame $\frF\defeq\langle W,B\rangle$, a \emph{valuation} function is any function $V\colon\Var\cup\Nom\to2^W$ such that for every nominal $i$, $V(i)$ is a singleton subset of the universe.
\end{definition}

Recall that the semantics of the at operator is given by the following:
\begin{equation}\tag{$\dftt{\at_i}$}
\frM,w\Vdash\at_i\varphi\Iffdef\frM,V(i)\Vdash\varphi\,.
\end{equation}
We can see that:
\begin{equation*}
\begin{split}
    \frM,w\Vdash\Bmod(i,j)\Iff{}&\\ (\exists x,y\in W)\,&(V(i)=\{x\}\tand V(j)=\{y\}\tand B(x,w,y))\,.
\end{split}
\end{equation*}

To simplify things, sometimes $V(i)$ is identified with its element and it is written, e.g., $B(V(i),w,V(j))$. Thus we can express the above equivalence as
\begin{equation*}
    \frM,w\Vdash\Bmod(i,j)\Iff B(V(i),w,V(j))\,.
\end{equation*}
We will also omit the letter `$\frM$' and write, e.g., $w\Vdash\varphi$ instead, if it is clear from the context or irrelevant what model is taken into account. Obviously:
\[
w\Vdash\at_i\Bmod(j,k)\Iff V(i)\Vdash\Bmod(j,k)\Iff B(V(j),V(i),V(k))\,.
\]

In a couple of places, we will also consider the language $\Hbatex$, which is an expansion of $\Hbat$ with the existential modality $\Ex$
\begin{equation}\tag{$\dftt{\Ex}$}
    w\Vdash\Ex\varphi\Iffdef(\exists x\in W)\,x\Vdash\varphi\,.
\end{equation}

\section{Non-definability}\label{sec:non-definability}

\subsection{The tools}

We are going to show that those frame properties that correspond to axioms \eqref{B1},\eqref{B3}--\eqref{B6} are not $\calL_{\beta}$-definable, and that density---expressed by \eqref{B8}---is not $\Hbat$-definable. For the latter, we generalize the theorem below to Kripke frames of arbitrary arity.

\begin{theorem}[{\citealp{tenCate-MTFEML}}]\label{th:tenCate-at-definable}
        A~first-order definable class of binary Kripke frames is $\calH(\at)$-definable iff it is closed under taking ultrafilter morphic images and generated subframes.
\end{theorem}
\noindent Above, $\calH(\at)$ is defined as
\[
\varphi\defeq \top\mid p\mid i\mid \neg\varphi\mid\varphi\wedge\psi\mid\Diamond\varphi\mid\at_i\varphi\,.
\]

The well-known notions of a \emph{bounded morphism} and a \emph{generated subframe} for binary relations have natural generalizations to $n$-ary ones.

\begin{definition}
    If $\langle W,R\rangle$ and $\langle W',R'\rangle$ are $n$-frames, then a mapping $f\colon W\to W'$ is a \emph{bounded morphism} if
    \begin{enumerate}[label=(\arabic*),itemsep=0pt]
        \item if $R(x_1,\ldots,x_n)$, then $R'(f(x_1),\ldots,f(x_n))$ (i.e,. $f$ preserves $R$, i.e., satisfies the forth condition),
        \item if $R'(f(w),x_1',\ldots,x_{n-1}')$, then there are $x_1,\ldots,x_{n-1}\in W$ such that for all $i\leqslant n-1$, $f(x_i)=x_i'$ and $R(w,x_1,\ldots,x_{n-1})$ (i.e., $f$ reflects $R'$, i.e., $f$ satisfies the back condition).\footnote{\citep[see p.\,140]{Blackburn-et-al-ML}}
    \end{enumerate}
\end{definition}

\begin{definition}
    Two $n$-frames $\frF_1\defeq\langle W_1,R_1\rangle$ and $\frF_2\defeq\langle W_2,R_2\rangle$ are \emph{disjoint} iff $W_1\cap W_2=\emptyset$. An indexed class $\{\frF_i\}_{i\in I}$ of frames is \emph{disjoint} iff its elements are pairwise disjoint. If $\{\frF_i\}_{i\in I}$ is a disjoint class of frames of the same modal similarity type $\tau$, then their \emph{disjoint union} is the frame $\biguplus_{i\in I}\frF_i\defeq\left\langle\bigcup_{i\in I}U_i,\bigcup_{i\in I}R_i\right\rangle$.
\end{definition}

\begin{definition}
    $\frF'\defeq\langle W',R'\rangle$ is a \emph{generated subframe} of $\frF\defeq\langle W,R\rangle$ (in symbols: $\frF'\gsub\frF$)  if $W'\subseteq W$, $R'=R\cap W^n$ (i.e., $\frF'$ is a subframe of $\frF$) and:
    \[
        \text{if}\ y\in W'\tand R(y,x_1,\ldots,x_{n-1})\,,\ \text{then}\ x_1,\ldots,x_{n-1}\in W'\,.
    \]
\end{definition}

For generated submodels of the language $\calH_\tau(\at)$, where $\tau$ is an arbitrary modal similarity type, we need a slight modification in comparison to the standard definition for modal languages.

\begin{definition}
    $\frM'\defeq\langle W',R',V'\rangle$ is a \emph{generated submodel} of $\frM\defeq\langle W,R,V\rangle$ (in symbols: $\frM'\gsub\frM$) iff $\langle W',R'\rangle$ is a generated subframe of $\langle W,R\rangle$ and
    \begin{enumerate}[label=(\arabic*),itemsep=0pt]
        \item for every $p\in\Var$, $V'(p)=V(p)\cap W'$,
        \item $V[\Nom]\subseteq W'$ and for every $i\in\Nom$, $V'(i)=V(i)$.
    \end{enumerate}
\end{definition}

\begin{definition}
    Given an $n$-frame $\frF\defeq\langle W,R\rangle$ we define an operation
    \[
    m_R\colon (2^W)^{n-1}\to2^W
    \]
    such that
    \begin{equation*}
    \begin{split}
    m_R(X_1,\ldots,X_{n-1})\defeq\{w\in W\mid\text{there are}\ w_1\in X_1,\ldots,w_{n-1}\in X_{n-1}\\
    \text{such that}\ R(w,w_1,\ldots,w_{n-1})\}\,.
    \end{split}
    \end{equation*}
    An \emph{ultrafilter extension} of an n-frame $\frF\defeq\langle W,R\rangle$ is the frame
    \[
    \ue\frF\defeq\langle\Ult(W),\Rue\rangle
    \]
    such that $\Ult(W)$ is the set of all ultrafilters of the power set algebra of $W$ and
    \[
    \mathord{\Rue}(\ult,\ult_1,\ldots,\ult_{n-1})\Iffdef m_R[\ult_1\times\ldots\times\ult_{n-1}]\subseteq\ult\,.
    \]
\end{definition}

\begin{definition}
    A frame $\frG$ is an \emph{ultrafilter bounded morphic image} of $\frF$ iff there is a surjective bounded morphism $f\colon\frF\to\ue\frG$ such that for every principal ultrafilter $\ult$ in $\ue\frG$ it is the case that $|f^{-1}(\{\ult\})|=1$.\footnote{The condition $|f^{-1}(\{\ult\})|=1$ says that $f$ is injective w.r.t. principal ultrafilters, and the idea is that such ultrafilters correspond to nominals.}
\end{definition}

The following seminal theorem will be applied to show non-definability of \eqref{B1}, \eqref{B3}--\eqref{B6}.

\begin{theorem}[{\citealp{Goldblatt-Thomason-ACIPML}}]\label{th:GT}
    Let $\tau$ be a modal similarity type. A~first-order definable class of $\tau$ frames is possibility definable iff it is closed under taking bounded morphic images, generated subframes, disjoint unions, and reflects ultrafilter extensions.
\end{theorem}

As mentioned in the opening paragraph of this section, the following generalization of Theorem~\ref{th:tenCate-at-definable} will be used to show non-definability of density in the extended language.

\begin{theorem}\label{th:Goldblatt:Thomason:Hybrid}
    Let $\tau$ be a modal similarity type, and $\calH_\tau(\at)$ be hybrid languages built with $\tau$. A~first-order definable class of Kripke $\tau$-frames is $\calH_\tau(\at)$-definable by arbitrary formulas iff it is closed under taking ultrafilter morphic images and generated subframes.
\end{theorem}

We will proceed in stages.

\begin{lemma}\label{lem:g-submodels-for-H-at}
    Let $\tau$ be a modal similarity type. If $\frM'\gsub\frM$ and $\varphi$ is a formula of $\calH_\tau(\at)$, then for every $w\in W'$\/\textup{:} $\frM',w\Vdash\varphi$ iff $\frM,w\Vdash\varphi$.
\end{lemma}
\begin{proof}
    The inductive proof is a combination of the well-known results for the standard modal language and for nominals and $\at$ operator.
\end{proof}

\begin{lemma}\label{lem:g-subframe:for:H:at}
    Let $\tau$ be a modal similarity type. If $\frF'\gsub\frF$ and $\varphi$ is a formula of $\calH_\tau(\at)$, then $\frF\Vdash\varphi$ entails $\frF'\Vdash\varphi$. So, $\calH_\tau(\at)$-validity is preserved under taking generated subframes.
\end{lemma}
\begin{proof}
    Assume that $\frF'\nVdash\varphi$, i.e., there are a model $\frM'$ and a world $s\in W'$ such that $\frM',s\nVdash\varphi$. Consider $V\colon\Var\cup\Nom\to2^W$ such that
    \begin{enumerate}[label=(\arabic*),itemsep=0pt]
        \item Put $V(p)\defeq V'(p)$. Since $V(p)\subseteq W'$, we clearly have $V'(p)=V(P)\cap W$.
        \item For every $i\in\Nom$, $V(i)=V'(i)$.
    \end{enumerate}
Thus $\frM'\gsub\frM$ and by Lemma~\ref{lem:g-submodels-for-H-at}, $\frM,s\nVdash\varphi$, so $\frF\nVdash\varphi$.
\end{proof}

Since the satisfaction operators can be characterized by $\Ex$ in the following way
\[
\at_ip\iffdef \Ex(i\wedge p)
\]
the proof of the next lemma can be limited to considering just $\calH_{\tau}(\Ex)$-formulas. We are going to omit the proof anyway, as it would be a mere repetition of the proof of \citep[Proposition 4.2.6]{tenCate-MTFEML}, with a straightforward adaptation of $\calH(\Ex)$-bisimulations to arbitrary modal similarity types.

\begin{lemma}\label{lem:ultra:image:HE}
    Let $\tau$ be a modal similarity type. $\calH_\tau(\at)$- and $\calH_\tau(\Ex)$-validity are preserved under taking ultrafilter morphic images.
\end{lemma}

We are now ready to prove Theorem \ref{th:Goldblatt:Thomason:Hybrid}.
The left-to-right direction is easy. From lemmas \ref{lem:g-subframe:for:H:at} and \ref{lem:ultra:image:HE}, if an elementary class $\Klass$ of Kripke $\tau$-frames is $\calH_\tau(\at)$-definable, then $\Klass$ is closed under taking ultrafilter morphic images and generated subframes.

\smallskip

For the other direction, we have to get our hands dirty. In the proof, we follow the strategy of \cite{tenCate-MTFEML} adapting the techniques to an arbitrary modal type~$\tau$. Let us begin with fixing an elementary class $\Klass$ of $\tau$-frames that is closed under taking ultrafilter morphic images and generated subframes. Let $\Th(\Klass)$ be the set of $\calH_\tau(\at)$-formulas valid in $\Klass$ (i.e., valid in every frame in $\Klass$). Suppose that $\frF$ is a $\tau$-frame with domain $W$ such that $\frF\Vdash \Th(\Klass)$. We will show that $\frF\in\Klass$.

To this end, we specify the following propositional variables and nominals:
\begin{enumerate}[label=(\roman*),itemsep=0pt]
\item With each $A\subseteq W$---$W$ being the domain of $\frF$---we associate a propositional variable $p_A$.
\item For every $w\in W$, we introduce a nominal $i_w$.
\end{enumerate}

We define $\bDelta$ to be the set of the following formulas for all $A\subseteq W$, $w\in W$ and $\triangle\in\tau$ ($\rho(\triangle)$ is the arity of $\triangle$):
\begin{align*}
p_{W\setminus A}&{}\iff\neg p_A\,,\\
p_{A\cap B}&{}\iff p_A\land p_B\,,\\
p_{m_{\triangle}(A_1,\ldots,A_{\rho(\triangle)})}&{}\iff\triangle(p_{A_1},\ldots,p_{A_{\rho(\triangle)}})\,,\\
i_{w}&{}\iff p_{\{w\}}\,.
\end{align*}
Until now, all our steps are mere repetitions of the steps of ten Cate's proof. The key difference appears below, as we consider modalities of arbitrary types.

For any possibility operator $\triangle\in\tau$ of arity $\rho(\triangle)$ and for any $i\leqslant\rho(\triangle)$ we define a unary necessity operator
\[
\Box_{i}^{\triangle}\varphi\defeq\neg\triangle(\top,\ldots,\top,\neg\varphi,\top,\ldots,\top)\,,
\]
where $\neg\varphi$ is put at the $i$-th argument place. Thus, every $\triangle\in\tau$ gives rise to $\rho(\triangle)$-many unary necessity operators that reach exactly one node (accessible from the current world) irrespective of the coordinate of the operator. Let
\[
B\defeq\left\{\Box_{i}^{\triangle}\suchthat \triangle\in\tau\  \text{and}\ i\leqslant\rho(\triangle)\right\}\,.
\]
With this we define
\[
\begin{split}
\bDelta_{\frF}\defeq\{\at_{i_v}\delta\mid&{}v\in W,\delta\in\bDelta\}\cup{}\\&\{\at_{i_v}f_1\ldots f_n\delta\mid v\in W, \delta\in\bDelta, f_1,\ldots,f_n\in B\mbox{ and }n\in\omega\}\,,
\end{split}
\]
By the natural valuation $N$ that sends $p_A$ to $A$ and $i_w$ to $\{w\}$, $\bDelta_{\frF}$ is satisfiable on $\frF$ at any point. It follows from the fact that in the model $\langle\frF,N\rangle$, every formula from $\bDelta$ is globally true. So for any formula from $\bDelta_{\frF}$, from whichever world $w$ we start, following the path determined by the sequence of unary boxes, at the end, we reach a world in which $\delta$ must be satisfied (all the $\delta$'s are globally true for the natural valuation).

Thus, $\bDelta_{\frF}$ is satisfiable on $\frF$, yet we do not know if $\frF$ is in $\Klass$. However, we can prove that

\begin{lemma}
$\bDelta_{\frF}$ is satisfiable on $\Klass$.\footnote{This is a generalization of \citep[Claim 1, p.\,56]{tenCate-MTFEML}}.
\end{lemma}

\begin{proof}
Since $\Klass$ is elementary and each formula in $\bDelta_\frF$ has a first-order counterpart by the standard translation (we speak about satisfiability, not about validity, so we stay on the first-order level), we can reduce the argument that $\bDelta_{\frF}$ has a model in $\Klass$ to the standard model-theoretic argument from compactness for first-order structures. Therefore, it suffices to show that every finite conjunction $\kappa$ of elements of $\bDelta_{\frF}$ is satisfiable on $\Klass$. Take the natural valuation $N$ from the paragraph above. Since by the assumption $\frF\Vdash\Th(\Klass)$, we obtain that $\langle\frF,N\rangle\Vdash \Th(\Klass)\cup\{\kappa\}$. Therefore $\neg\kappa\notin \Th(\Klass)$, so $\kappa$ is satisfiable on~$\Klass$.
\end{proof}

Thus, there is a $\frG\in\Klass$ such that $\bDelta_{\frF}$ is globally true at $\langle\frG,V\rangle$ for some $V$ on $\frG$ (globally since all formulas in $\bDelta_{\frF}$ begin with $\at_i$ operators). Since $\Klass$ is closed under taking generated subframes, the frame $\frG'$ generated by the set $\bigcup V[\Omega]$ of points in $\frG$ named by a nominal is in $\Klass$. Therefore, the domain $W'$ of $\frG'$ consists of all the worlds that are accessible from named worlds via a~finite number of $R_{\triangle}$-steps (for each $\triangle\in\tau$) or just by ``jumps'' via $\at_i$-operators. Clearly, $\bDelta_{\frF}$ is globally true at $\langle\frG',V'\rangle$ where $V'$ is a restriction of $V$ to $W'$. Since all points of $\frG'$ are reachable from the named points, we have that $\bDelta$ is globally true in the model.

Take $\langle\frG^{\ast},V^{\ast}\rangle$ to be an $\omega$-saturated elementary extension of $\langle\frG',V'\rangle$. By elementarity, $\frG^{\ast}\in\Klass$ and
$\langle\frG^{\ast},V^{\ast}\rangle$ globally satisfies $\bDelta$.

\begin{lemma}
$\frF$ is an ultrafilter bounded morphic image of $\frG^*$.\footnote{This is a generalization of \citep[Claim 2, p.\,56]{tenCate-MTFEML}.}
\end{lemma}

\begin{proof}
For brevity, let $\frM^\ast\defeq\langle\frG^{\ast},V^{\ast}\rangle$. For any $v\in\frG^*$, let
\[
f(v)\defeq\{A\subseteq W\mid\frM^{\ast},v\Vdash p_A\}\,.
\]
We will show that $f$ is a surjective bounded morphism from $\frG^*$ onto $\mathfrak{ue}\frF$ and $|f^{-1}(\ult)|=1$ for all principal ultrafilters $\ult\in\mathfrak{ue}\frF$. We divide the proof into five steps.

\smallskip

(1st step) $f(v)$ is an ultrafilter on $\frF$. This follows from the fact that $\bDelta$ is globally satisfied in the model $\frM^{\ast}$; for every world $w$, $\frM^{\ast},w\nVdash p_\emptyset$ and $\frM^{\ast},w\Vdash p_W$; for any set $A\subseteq W$ and every world $w$, exactly one of $\frMast,w\Vdash p_A$ and $\frMast,w\Vdash \neg p_A$ holds; for any sets $A,B\subseteq W$ and every world $w$, $\frMast,w\Vdash p_{A\cap B}\leftrightarrow p_A \wedge p_B$.

\smallskip

(2nd step) $f$ is surjective. It suffices to prove that for any ultrafilter $\ult\in\ue\frF$, $\{p_A\mid A\in\ult\}$ is satisfiable on $\frMast$. For this, on the other hand, it is enough to show finite satisfiability as $\frM^\ast$ is $\omega$-saturated. To this end, observe that for finitely many $A_1,\ldots,A_n\in\ult$ we have that $\bigcap_{i=1}^n A_i\in\ult$, and so there is a $v$ in the intersection: $v\in\bigcap_{i=1}^{n}A_i$. We have that
    \[
        i_v\iff p_{\{v\}}\iff p_{\{v\}\cap\bigcap_{i=1}^{n}A_i}\iff p_{\{v\}}\wedge p_{\bigcap_{i=1}^{n}A_i}\,.
    \]
Therefore $\bDelta\Vdash i_v\rightarrow p_{\bigcap_{i=1}^{n}A_i}$. For $\bDelta_{i_{v}}\defeq\{\at_{i_{v}}\delta\mid\delta\in\bDelta\}$ it follows that $\bDelta_{i_{v}}\Vdash\at_{i_{v}}(i_{v}\rightarrow p_{\bigcap_{i=1}^{n}A_i})$. But $\at_{i_v}i_v$ is globally true, as is the formula $\at_{i_{v}}(p\rightarrow q)\rightarrow(\at_{i_{v}}p\rightarrow\at_{i_{v}}q))$. So in consequence, we get that $\bDelta_{i_{v}}\Vdash\at_{i_{v}} p_{\bigcap_{i=1}^{n}A_i}$. As $\bDelta_{i_{v}}\subseteq\bDelta_{\frF}$, we obtain $\bDelta_{\frF}\Vdash\at_{i_{v}} p_{\bigcap_{i=1}^{n}A_i}$. Then $\bDelta_{\frF}\Vdash\at_{i_v}p_{\bigcap_{i=1}^{n} A_i}$, which entails that $\frMast\Vdash\at_{i_v}p_{\bigcap_{i=1}^{n} A_i}$. In consequence, there's a $u\in\frGast$ such that for every $A\in\ult$, $\frMast,u\Vdash p_A$, and so $f(u)=\ult$, as required.\footnote{To be precise, we first translate every propositional variable $p_A$ to $P_A(x)$ (where $P_A$ is a unary predicate and $x$ is a first-order variable), and then we apply $\omega$-saturatedness to the set $\{P_A(x)\mid A\in\ult\}$.}

\smallskip

(3rd step) The forth condition for $f$ holds, i.e.,
\[
R^*_{\triangle}(v,v_1,\ldots,v_{\rho(\triangle)})\ \text{implies}\ R^{\mathfrak{ue}}_{\triangle}(f(v),f(v_1),\ldots,f(v_{\rho(\triangle)}))\,.
\]
Proving this amounts to showing that if $R^*_{\triangle}(v,v_1,\ldots,v_{\rho(\triangle)})$, then
\[
m_{\triangle}[f(v_1)\times\ldots\times f(v_{\rho(\triangle)})]\subseteq f(v)
\]
To this end, suppose that $A_1\in f(v_1),\ldots,A_{\rho(\triangle)}\in f(v_{\rho(\triangle)})$, i.e., $\frMast,v_i\Vdash p_{A_i}$, for every $i\leqslant \rho(\triangle)$. Since $R^*_{\triangle}(v,v_1,\ldots,v_{\rho(\triangle)})$, we have that $\frMast,v\Vdash\triangle(p_{A_1},\ldots,p_{A_{\rho(\triangle)}})$. As $\bDelta$ is globally true, $\frMast,v\Vdash p_{m_{\triangle}(A_1,\ldots,A_{\rho(\triangle)})}$, therefore $m_{\triangle}(A_1,\ldots,A_{\rho(\triangle)})\in f(v)$.

\smallskip

(4th step) For the back condition, we have to prove that
\[
R^{\mathfrak{ue}}(f(v),\ult_1,\ldots, \ult_{\rho(\triangle)})
\]
entails existence of $v_1,\ldots,v_{\rho(\triangle)}\in\frG^*$ such that
\[
\text{for all}\ i\leqslant\rho(\triangle)\,,\  f(v_i)=\ult_i\quad\text{and}\quad R^*_{\triangle}(v,v_1,\ldots,v_{\rho(\triangle)})\,.
\]
For this, it is enough to find $v_1,\ldots,v_{\rho(\triangle)}$ such that $R^*_{\triangle}(v,v_1,\ldots, v_{\rho(\triangle)})$ and each $v_i$ satisfies $\{p_A\mid A\in\ult_i\}$ for $i\leqslant\rho(\triangle)$. By $\omega$-saturatedness\footnote{See Definitions 2.53, 2.63 and Theorem 2.65 in \citep{Blackburn-et-al-ML}.}, it suffices to show that for any finite set $B_i\subseteq\{p_A\mid A\in\ult_i\}$, there are $v_1,\ldots,v_{\rho(\triangle)}$ such that $R^*_{\triangle}(v,v_1,\ldots, v_{\rho(\triangle)})$ and each $v_i$ satisfies $B_i$. Take any $A_{i,1},\ldots,A_{i,n_i}\in \ult_i$. Then
\[
\bigcap_{j=1}^{n_i}A_{i,j}\in\ult_i\,,\quad\text{and so}\quad \mu\defeq m_{\triangle}\left(\bigcap_{j=1}^{n_1}A_{1,j},\ldots,\bigcap_{j=1}^{n_{\rho(\triangle)}}A_{\rho(\triangle),j}\right)\in f(v)\,.
\]
By definition $\frMast,v\Vdash p_{\mu}$. By global truth of $\bDelta$, there are $v_1,\ldots,v_{\rho(\triangle)}$ such that $R^{\ast}_{\triangle}(v,v_1,\ldots, v_{\rho(\triangle)})$ and each $v_i$ satisfies $\left\{p_{A_{i,1}},\ldots,p_{A_{i,n_{i}}}\right\}$.

\smallskip

(5th step) $|f^{-1}(\ult)|=1$ for all principal ultrafilters $\ult\in\mathfrak{ue}\frF$. Suppose that $f(x)=f(y)=\{A\subseteq W\mid w\in A\}$ for some $x,y\in\frG^*$ and $w\in \frF$. By definition, $x$ and $y$ satisfy $p_{\{w\}}$. By global truth of $\bDelta$, $x$ and $y$ are both named by the nominal $i_{w}$, which implies that $x=y$.
\end{proof}
Since $\Klass$ is closed under ultrafilter morphic images, we obtain $\frF\in\Klass$. This completes the proof of the left-to-right direction, and---in consequence---of the whole Theorem~\ref{th:Goldblatt:Thomason:Hybrid}.

\subsection{The results}

We are now applying the outcomes of the previous section to show non-definability of certain classes of frames related to the betweenness axioms. For heuristic reasons, we distinguish the middle projection of $B(x,y,z)$, i.e., we interpret this as \emph{$y$ is between $x$ and $z$}. For this reason, the second point in the definition  of the bounded morphism for the particular case of $B$ takes the following form
\begin{enumerate}[label=(\arabic*),itemsep=0pt,start=2]
        \item if $B'(x',f(y),z')$, then there are $x,z\in W$ such that $f(x)=x$, $f(z)=z$ and $B(x,y,z)$.
    \end{enumerate}

\begin{theorem}\label{th:non-definability}
    For every $i\in\{1,3,4,\ldots,6\}$ the class of frames that satisfies \textup{(Bi)} is not $\calL_{\beta}$-definable. Moreover, the class of \eqref{B8}-frames is not $\Hbat$-definable.
\end{theorem}
\begin{proof}
    \eqref{B1} Consider a 3-frame $\frF\defeq\langle U,B\rangle$ such that $U\defeq\{x,y,z\}$ and $B$ consists of all proper triples of $U^3$ (i.e., those $\langle a,b,c\rangle$ for which $\#\{a,b,c\}$). Let $\frW\defeq\langle W,R\rangle$ be a 3-frame such that $W\defeq\{w\}$ and $R\defeq\{\langle w,w,w\rangle\}$. It is routine to check that the constant mapping $f\colon U\to W$ is a bounded morphism. For example, for $R(w,f(x),w)$ we have that $B(y,x,z)$ and $f(y)=w=f(z)$. $\frF$ satisfies \eqref{B1} by the definition of $B$, but the axiom fails in $\frW$.

    \smallskip

    \eqref{B3} Take $\frF\defeq\langle \{x,y\},B\rangle$ with $B\defeq\{\langle x,x,y\rangle,\langle y,y,x\rangle\}$. Let $\frF'$ be the frame $\langle\{a\},\{\langle a,a,a\rangle\}\rangle$, in which \eqref{B3} fails in obvious way. It is routine to verify that the unique mapping $f\colon \{x,y\}\to\{a\}$ is a bounded morphism.

   \smallskip

    \eqref{B4} Take any countable set $U$ and partition it into four sets: $\{x_n\mid n\in\omega\}$, $\{y_n\mid n\in\omega\}$, $\{z_n\mid n\in\omega\}$, and $\{u_n\mid n\in\omega\setminus\{0\}\}$ (we shift the numbering for the sake of a clearer presentation). Define the following relation:
    \begin{align*}
        B(a,b,c)\Iffdef &(\exists n\in\omega)\,(a=x_n\tand b=y_n\tand c=z_n)\tor\\
        &(\exists n\in\omega)\,(a=y_{n+1}\tand b=z_n\tand c=u_{n+1})\,.
    \end{align*}
    Observe that $\frF\defeq\langle U,B\rangle$ is a 3-frame in which it can never be the case that $B(a,b,c)\wedge B(b,c,d)$. Indeed, if $B(x_n,y_n,z_n)$, then $\neg B(y_n,z_n,d)$ for any~$d\in U$. Similarly, if $B(y_{n+1},z_n,u_{n+1})$, then for no $d\in U$, $B(z_n,u_{n+1},d)$. Thus $\frF$ satisfies \eqref{B4} vacuously. Now, consider $\frW\defeq\langle W,B'\rangle$ such that $W\defeq\{x,y,z,u\}$ and $B'\defeq\{\langle x,y,z\rangle,\langle y,z,u\rangle\}$. It is clear that $\frW$ does not satisfy \eqref{B4}, so what remains to be done is to define a bounded morphism $f\colon U\to W$. To this end for any $n\in\omega$, let $f(a_n)\defeq a$ for each $a\in\{x,y,z\}$ and for any $n\in\omega\setminus\{0\}$, let $f(u_n)\defeq u$. For the forth condition, if $B(x_n,y_n,z_n)$, then clearly $B'(f(x_n),f(y_n),f(z_n))$, as the latter is simply $B'(x,y,z)$; and if $B(y_{n+1},z_n,u_{n+1})$, then $B'(f(y_{n+1}),f(z_n),f(u_{n+1}))$. For the back condition, we have two possibilities. In the first, $B'(x,f(a),z)$, which means that $f(a)=y$, and so for some $n\in\omega$, $a=y_n$. But then $f(x_n)=x$ and $f(z_n)=z$ and $B(x_n,y_n,z_n)$. In the second, $B'(y,f(a),u)$ and thus $f(a)=z$, i.e., for an $n\in\omega$, $a=z_n$. So now: $f(y_{n+1})=y$, $f(u_{n+1})=z$ and $B(y_{n+1},z_n,u_{n+1})$.

        \begin{figure}
        \centering
        \begin{tikzpicture}
            \foreach \x\n in {0/2,1.5/1,3/0}
                {
                    \node[fill,circle,inner sep=2pt] (x\n) at (\x,3) {};
                    \node[above right] at (x\n) {$x_{\n}$};
                }
            \foreach \y\n in  {-1.5/3,0/2,1.5/1,3/0}
                {
                     \node[fill,circle,inner sep=2pt] (y\n) at (\y,1.5) {};
                     \node[above right] at (y\n) {$y_{\n}$};
                }
            \foreach \z\n in  {0/2,1.5/1,3/0}
                {
                     \node[fill,circle,inner sep=2pt] (z\n) at (\z,0) {};
                     \node[above right] at (z\n) {$z_{\n}$};
                }
            \foreach \u\n in  {1.5/3,3/2,4.5/1}
                {
                     \node[fill,circle,inner sep=2pt] (u\n) at (\u,-1.5) {};
                     \node[above right] at (u\n) {$u_{\n}$};
                }
            \foreach \n in {0,1,2}
                {
                \draw (x\n) -- (y\n) -- (z\n);
                }
                \draw[dashed] (y1) -- (z0) -- (u1) (y2) -- (z1) -- (u2) (y3) -- (z2) -- (u3);
                \draw[dotted,shorten >=3pt] (-3,3) -- (x2);
                \draw[dotted,shorten >=3pt] (-3,1.5) -- (y3);
                \draw[dotted,shorten >=3pt] (-3,0) -- (z2);
                \draw[dotted,shorten >=3pt] (-3,-1.5) -- (u3);

                \draw[rounded corners,gray,mygold] (-3,3.5) -- (3.75,3.5) -- (3.75,2.5) -- (-3,2.5);
                \draw[rounded corners,gray,yshift=-1.5cm,mygold] (-3,3.5) -- (3.75,3.5) -- (3.75,2.5) -- (-3,2.5);
                \draw[rounded corners,gray,yshift=-3cm,mygold] (-3,3.5) -- (3.75,3.5) -- (3.75,2.5) -- (-3,2.5);
                \draw[rounded corners,gray,yshift=-4.5cm,xshift=1.5cm,mygold] (-4.5,3.5) -- (3.75,3.5) -- (3.75,2.5) -- (-4.5,2.5);

                \foreach \a\c in  {x/3,y/1.5,z/0,u/-1.5}
                    {
                        \node[fill,circle,inner sep=2pt] (\a) at (8,\c) {};
                        \node[above left] at (\a) {$\a$};
                    }

                    \draw (8,0.75) e llipse (1cm and 3cm);

                \foreach \a\b in {3/x,1.5/y,0/z}
                    {
                        \draw [->,shorten >=2pt,myteal] (3.75,\a) -- node [midway,above] {$f$} (\b);
                    }
                \draw [->,shorten >=2pt,myteal] (5.25,-1.5) -- node [midway,above] {$f$} (u);

            \draw [mycoral] (x) -- (y) -- (z);
            \draw [dashed,mycoral] (y) to [out=315,in=45] (z) to [out=315,in=45] (u);
        \end{tikzpicture}
        \caption{A bounded morphism $f$ which shows that outer transitivity for $B$ (expressed by \eqref{B4}) is not modally definable. Betweenness relations hold between triples of points that are connected by the lines of the same style. The domain on the left is divided into four clusters. Each cluster is mapped onto precisely one point.}
        \label{fig:enter-label}
    \end{figure}
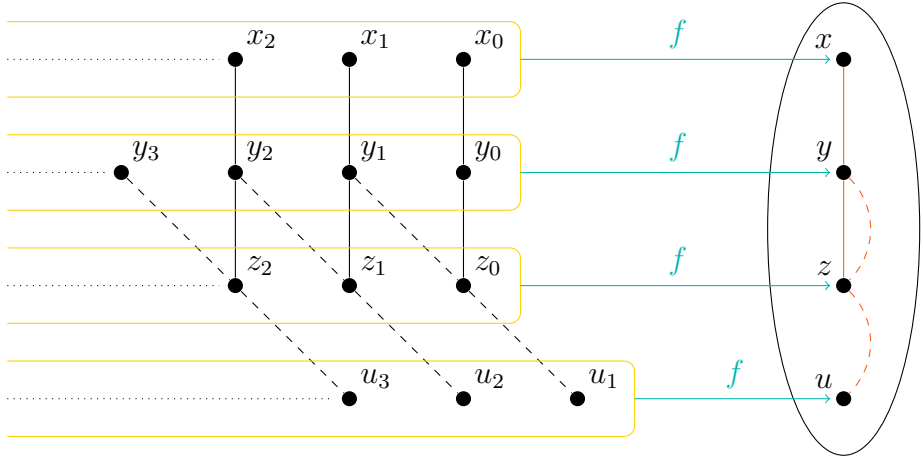

    \smallskip

    \eqref{B5} For this axiom we consider the frame $\frF\defeq\langle U,B\rangle$ with the domain $U\defeq\{x_0,y_0,u_0,z_0,z_1\}$, and $B\defeq\{\langle x_0,y_0,z_0\rangle, \langle y_0,u_0,z_1\rangle\}$; and another frame $\frW\defeq\langle W,B'\rangle$ in which $W\defeq\{x,y,z,u\}$ and $B'\defeq\{\langle x,y,z\rangle,\langle y,u,z\rangle\}$. Clearly, the first frame does not meet the antecedent of \eqref{B5}, so it satisfies the axiom vacuously, while the second frame meets the antecedent only.

    Take $f\colon U\to W$ such that $f(z_i)\defeq z$ for $i\in\{0,1\}$, and $f(x_0)\defeq x$, $f(y_0)\defeq y$ and $f(u_0)\defeq u$. The forth condition is clearly true about $f$, and for the back if $B'(a,f(w),b)$, then we have two possibilities. In the first, $a=x$, $f(w)=y$ and $b=z$, which means that $w=y_0$. So $f(x_0)=a$ and $f(z_0)=b$ and $B(x_0,w,z_0)$. In the second, $a=y$, $f(w)=u$ and $b=z$. This time, $w=u_0$, $f(y_0)=a$, $f(z_1)=b$ and $B(y_0,w,z_1)$.

    \smallskip

    \eqref{B6} For linearity we are going to use the disjoint unions method. Let us notice that \eqref{B6} considered in isolation as the linearity axiom has the following form (equivalent to \eqref{B6} in presence of \eqref{B2})
    \[
    \begin{split}
    \#(x,&y,z)\rarrow \\&B(x,y,z)\vee B(x,z,y)\vee B(y,x,z)\vee B(y,z,x)\vee B(z,x,y)\vee B(z,y,x)\,.
    \end{split}
    \]
    To show this property is not modally definable, take the following three-element frame  $\frF\defeq\langle \{a,b,c\},\{\langle a,b,c\rangle\}\rangle$, which clearly satisfies the formula above. Now, we take the disjoint copy $\frF'=\langle \{a',b',c'\},\{\langle a',b',c'\rangle\}\rangle$ of $\frF$, and the disjoint union $\frF\uplus\frF'$. The linearity axiom fails in the union, as, e.g., for $\{a,b,c'\}$, no permutation of the set is in the union of the two relations.

    \smallskip

    \eqref{B8} We will show that generated subframes do not preserve density. To this end consider the frame $\frF\defeq\langle[0,1]^{\Ratio},B_<\rangle$ where  $[0,1]^{\Ratio}$ is a closed interval of rational numbers and $B_<$ is induced by the irreflexive order $<$ on the interval. The subframe $\frF'\defeq\langle\{0,1\},\emptyset\rangle$ is generated, since neither $0$ nor $1$ are between any numbers from the interval. But $\frF'$ fails to meet density in an obvious way, so the class of dense frames is not $\Hbat$-definable by Theorem~\ref{th:Goldblatt:Thomason:Hybrid}, neither possibility definable by Theorem~\ref{th:GT}.
\end{proof}

\pagebreak

\begin{corollary}
    $\DLBWE$ is not $\calL_{\beta}$-definable.
\end{corollary}
\begin{proof}
    If $\frF\in\DLBWE$, then the unique mapping onto the frame \[
    \frF'\defeq\langle\{x\},\{\langle x,x,x\rangle\}\rangle
    \]
    is a bounded morphism. Thus the class is not closed for bounded morphic images.
\end{proof}

\section{Definability in extended languages}\label{sec:definability}

Let us observe that all properties of betweenness expressed by axioms \eqref{B1}--\eqref{B7} can be captured directly by means of pure formulas in an extended language.
Recall that a \emph{pure} formula is a formula that is built from operators and nominals without any propositional variables.

And so, the betweenness axioms \eqref{B1}--\eqref{B6} have the following hybrid correspondents:
\begin{gather}
\at_i\Bmod(j,k)\rarrow\neg\at_ij\wedge\neg\at_ik\wedge\neg\at_kj\,,\tag{$\mathtt{HB1}$}\label{HB1}\\
\Bmod(i,j)\rarrow\Bmod(j,i)\,,\tag{$\mathtt{HB2}$}\label{HB2}\\
\at_j\Bmod(i,k)\rarrow\neg\at_k\Bmod(i,j)\,,\tag{$\mathtt{HB3}$}\label{HB3}\\
\at_j\Bmod(i,k)\wedge\at_k\Bmod (j,l)\rarrow\at_j\Bmod (i,l)\wedge\at_k\Bmod (i,l)\,,\tag{$\mathtt{HB4}$}\label{HB4}\\
\at_j\Bmod (i,k)\wedge\at_l\Bmod (i,j)\rarrow\at_l\Bmod (i,k)\wedge \at_j\langle B\rangle (l,k)\,,\tag{$\mathtt{HB5}$}\label{HB5}\\
\neg\at_ij\wedge\neg\at_ik\wedge\neg\at_kj\rarrow \at_j\Bmod(i,k)\vee\at_k\Bmod(i,j)\vee\at_i\Bmod(j,k)\,.\tag{$\mathtt{HB6}$}\label{HB6}
\end{gather}
\eqref{B7} may be expressed by means of a modal formula
\begin{equation}
\Bmod(\top,\top)\,.\tag{$\mathtt{HB7}$}\label{HB7}
\end{equation}
Density---expressed by \eqref{B8}---can be captured directly  by means of a formula of $\Hbatex$ language:
\begin{equation}
\neg\at_ij\rarrow\Ex\Bmod(i,j)\,.\tag{$\mathtt{HB8}$}\label{HB8}
\end{equation}
And, as is well-known, \eqref{HB2} has also an equivalent modal counterpart:
\begin{equation}
\Bmod(\varphi,\psi)\rarrow\Bmod(\psi,\varphi)\,.\tag{$\mathtt{HB2_m}$}\label{MB2}
\end{equation}
The verification is routine and we leave it to the reader.

Axioms \eqref{HB1}--\eqref{HB7} are so-called pure axioms. If we add them to the basic system of hybrid logic with $\Ex$, $\at$ and nominals, then we get completeness with respect to their defining properties automatically.

Although we have shown that \eqref{B8} is not $\Hbat$-definable by showing that the class of \eqref{B8}-frames is not closed under taking generated subframes, the proof uses bounded linear frame $[0,1]^{\Ratio}$. So it is possible that the class is relatively $\Hbat$-definable with respect to the class $\LBWE$. Indeed, since \eqref{B1}--\eqref{B7} are all $\Hbat$-definable, we have that

\begin{theorem}\label{th:density-definable-abstract}
$\DLBWE$ is $\Hbat$-definable iff \eqref{B8} is relatively $\Hbat$-definable with respect to $\LBWE$.
\end{theorem}
\begin{proof}
If $\DLBWE$ is $\Hbat$-definable, then its defining formulas are the defining formulas for \eqref{B8} with respect to the class  $\LBWE$. For the other direction, if \eqref{B8} is $\at$-definable with respect to $\LBWE$, then since \eqref{B1}--\eqref{B7} are all $\Hbat$-definable, the relative defining formulas for \eqref{B8} together with the defining formulas for \eqref{B1}--\eqref{B7} are the defining formulas for $\DLBWE$.
\end{proof}

\section{The axiom of density with respect to \texorpdfstring{$\LBWE$}{LBWE}}\label{sec:definability-density}

Up to now, we have only shown that there exists an axiomatization of the class $\DLBWE$ by means of formulas of the language $\Hbat$. Below, we pinpoint both the defining modal formula and the defining pure hybrid formula (even without $\at$) of density with respect to the class $\LBWE$.

Let us begin with an observation that the existential modality $\Ex$ is $\calL_{\beta}$-definable relative to the class of linear betweenness orders without endpoints.

\begin{theorem}\label{th:E-elimination-in-LBWE}
    If $\frF\in\LBWE$, then $\frF\Vdash\Ex\varphi\iff\Bmod(\varphi,\top)\vee\varphi$.
\end{theorem}
\begin{proof}
Suppose $\frM$ is a model based on $\frF$ and $w\in W$ is such that $\frM,w\Vdash\Ex\varphi$. This means that there is $x\in W$ such that $\frM,x\Vdash\varphi$. If $w=x$ we're done. So assume $w\neq x$ and $\frM,w\Vdash\neg\varphi$ (otherwise, we're done again). By Lemma~\ref{lem:exists-to-one-side} there is a point $u$ such that $B(x,w,u)$. So $\frM,w\Vdash\Bmod(\varphi,\top)$, and the conclusion follows.

The other direction is obvious.
\end{proof}

Thus, with respect to the class of linear b-frames without endpoints, we have the following equivalences
\[
\at_i\varphi\iff\Ex(i\wedge\varphi)\iff \Bmod(i\wedge\varphi,\top)\vee(i\wedge\varphi)\,.
\]
Therefore, with respect to this class, any property that is $\Hbat$-definable or $\Hbex$-definable must be definable by a formula with the modal operator only and possibly nominals.

In Theorem~\ref{th:density-definable-abstract} we have shown that the class $\DLBWE$ is $\Hbat$-definable relative to the class $\LBWE$. This abstract result can be made concrete by pointing to a modal axiom which is a relative defining formula for \eqref{B8}, and which is the missing piece (modulo translation) in our hybrid characterization of the class.

\begin{theorem}\label{th:relative-density}
For any frame $\frF\in\LBWE$\/\textup{:} $\frF\models\eqref{B8}$ iff $\frF\Vdash\Conv p\to
\Conv\Conv p$.\footnote{The idea for the theorem and the proof comes from \citep{Bezhanishvili-et-al-MLOB}.}
\end{theorem}

\begin{proof}
($\Larrow$) Suppose that $\frF$ has two distinct points~$x$ and~$y$ such that no point distinct from them is located between~them. By \eqref{B7} there are points $a$ and $b$ such that $B(a,x,b)$, and by \eqref{B1}, $\#(a,x,b)$. We have three cases to consider.

\smallskip

(1st) $b=y$: in which case $B(a,x,y)$. Take the valuation $V$ such that $V(p)\defeq\{a,y\}$. Thus $x\Vdash\Conv p$ but $\Conv\Conv p$ fails at $x$. To see this, assume for a contradiction that $x\Vdash\Conv\Conv p$.  Thus, there are $u$ and $w$ such that (i) $B(u,x,w)$ and $u\Vdash\Conv p$ and $w\Vdash\Conv p$. Further, these entail existence of $u_1,u_2,w_1$ and $w_2$ such that (ii) $B(u_1,u,u_2)$ and $\{u_1,u_2\}\subseteq V(p)$ and (iii) $B(w_1,w,w_2)$ and $\{w_1,w_2\}\subseteq V(p)$. In consequence
\[
\{u_1,u_2,w_1,w_2\}\subseteq\{a,y\}\,,
\]
so there are two possibilities
\[
u_1=w_1\tand u_2=w_2\quad\tor\quad u_1=w_2\tand u_1=w_2\,.
\]
Consider the first one, additionally assuming that
\[
u_1=w_1=a\qtand u_2=w_2=y\,.
\]
Gathering the assumptions, (i), (ii), (iii), and the above, we get that
\begin{equation*}
\text{(a)}\ B(a,x,y)\tand\text{(b)}\ B(u,x,w)\tand\text{(c)}\ B(a,u,y)\tand\text{(d)}\ B(a,w,y)\,.
\end{equation*}
From (a), (d), and Lemma~\ref{lem:proj-for-B} we obtain that either $B(a,x,w)$ or $B(a,w,x)$.

In the former case, from (a) and Lemma~\ref{lem:proj-for-B} we get that $B(x,y,w)$ or $B(x,w,y)$. Since the second possibility is excluded, the first one obtains. But then, using (d) and \eqref{B4}, we have $B(a,w,x)$, which stands in contradiction to $B(a,x,w)$ and \eqref{B3}.

In the latter case, from (d) and Lemma~\ref{lem:proj-for-B}, we get that $B(w,x,y)$ or $B(w,y,x)$. If $B(w,x,y)$, then we use (b) and apply Lemma~\ref{lem:proj-for-B} to obtain that either $B(x,y,u)$ or $B(x,u,y)$. The second case cannot hold, so $B(x,y,u)$. Using symmetric consequence of (c), i.e., $B(y,u,a)$ we get that $B(x,y,a)$, so $B(a,y,x)$, which contradicts (a). If $B(w,y,x)$, then $B(x,y,w)$. From (d) we have that $B(y,w,a)$, so $B(x,y,a)$ by \eqref{B4}. Thus $B(a,y,x)$ and again we have a contradiction with~(a).

The other cases are proved in an analogous way.

\smallskip

(2nd) $a=y$: in which case $B(y,x,b)$ and we proceed with the valuation $V$ such that $V(p)\defeq\{y,b\}$. The proof is analogous to the one above.

\smallskip

(3rd) $\#(a,b,y)$: in which case also $\#(a,x,y)$ and by \eqref{B6} we have that
    \[
        B(a,x,y)\quad\tor \quad B(a,y,x)\quad\tor\quad B(x,a,y)\,.
    \]
    The third case is excluded by the assumption, and in the remaining two we either fix $V(p)\defeq\{a,y\}$ to show that $\Conv p\rarrow\Conv\Conv p$ fails at $x$ in the first case, or we put $V(p)\defeq\{a,x\}$ to show that the condition is false at $y$ in the second case.

    \smallskip

    The geometrical idea behind this part of the proof is depicted in Figure~\ref{fig:density-axiom-1}.

\begin{figure}
    \centering
\begin{tikzpicture}
  \path (0,0) coordinate (x) (1.5,0) coordinate (y) (-2,0) coordinate (w) (3.5,0) coordinate (z);
  \draw (-3,0) -- (x) (y) -- (4,0);
  \foreach \x in {x,y,w} {\fill (\x) circle (2pt);}
  \node[above] at (x) {$x$} node[above] at (y) {$y$} node[above] at (w) {$w$};
  \node[below] at (y) {$p$} node[below] at (w) {$p$};
  \node[above=10pt] at (x) {$\Conv p$};
  \draw[snake=brace,mirror snake,raise snake=15pt] (-1.8,0) -- (0,0) node[midway,below=17pt] {$\Conv p$} (1.5,0) -- (4,0) node[midway,below=17pt] {$\neg\Conv p$};
\end{tikzpicture}\caption{Failure of the modal density axiom in a non-dense frame. Since there is no point between $x$ and $y$, $x\nVdash\Conv\Conv p$}\label{fig:density-axiom-1}
\end{figure}
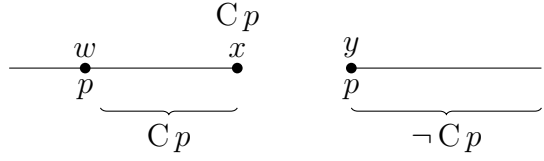

\smallskip

($\Rarrow$) Let now $x$ be a~point such that $x\Vdash\Conv p\wedge\neg\Conv\Conv p$. Since $\Conv p$ is true at $x$, there are points $y$ and $z$ such that $B(y,x,z)$ and $y,z\Vdash p$. We have two possibilities: either all points on the side of $y$ of $x$ satisfy $\neg\Conv p$, i.e.,
\[
\{u\in W\mid u=y\vee B(u,y,x)\vee B(y,u,x)\}\cap V(\Conv p)=\emptyset
\]
or all points on the other side of $x$ satisfy $\neg\Conv p$, i.e.,
\[
\{w\in W\mid w=z\vee B(x,w,z)\vee B(x,z,w)\}\cap V(\Conv p)=\emptyset\,.
\]
This can be proven by assuming that if both intersections are non-empty, then $x\Vdash\Conv\Conv p$. For example, if $u=y$ and $w=z$ and $u,w\in V(\Conv p)$, then $B(u,x,w)$ and $x\Vdash\Conv\Conv p$. If $u,w\in V(\Conv p)$ and $B(u,y,x)$ and $B(x,w,z)$, then together with $B(y,x,z)$ we have firstly that $B(y,x,w)$ by \eqref{B5}, and secondly that $B(u,x,w)$ by \eqref{B4'}. Thus $x\Vdash\Conv\Conv p$. The other seven possibilities are similar.

In the first case---when all points on the side of $y$ of $x$ satisfy $\neg\Conv p$---there cannot be any point between $y$ and $x$, since such point would be---by \eqref{B5'}---between $y$ and $z$, and therefore would have to satisfy $\Conv p$. The second case is analogous. Figure~\ref{fig:density-axiom-2} presents the idea of the proof in a geometrical way.\qedhere

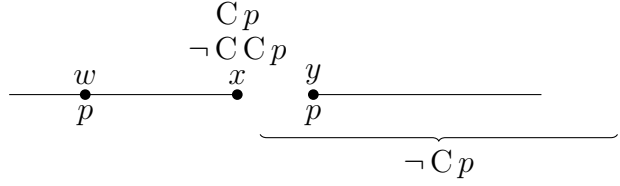
\begin{figure}
    \centering
\begin{tikzpicture}
  \path (0,0) coordinate (x) (1,0) coordinate (y) (-2,0) coordinate (z) (3.5,0) coordinate (z);
  \draw (-3,0) -- (x) (y) -- (4,0);
  \foreach \x in {x,y,w} {\fill (\x) circle (2pt);}
  \node[above] at (x) {$x$} node[above] at (y) {$y$} node[above] at (w) {$w$};
  \node[below] at (y) {$p$} node[below] at (w) {$p$};
  \node[above=20pt] at (x) {$\Conv p$} node[above=7pt] at (x) {$\neg\Conv\Conv p$};
  \draw[snake=brace,mirror snake,raise snake=15pt] (0.3,0) -- (5,0) node[midway,below=17pt] {$\neg\Conv p$};
\end{tikzpicture}\caption{Failure of the modal density axioms entails failure of the geometrical density. Here is pictured the situation in which all points on the $y$ side of $x$ fail to satisfy $\Conv p$.}\label{fig:density-axiom-2}
\end{figure}
\end{proof}

Since our intention is to apply the techniques of \cite[Section 7.3]{Blackburn-et-al-ML} to obtain completeness results, we need a system whose only non-standard modalities are satisfaction operators and whose all formulas are pure. Among the axioms \eqref{HB1}--\eqref{HB7} the only one that does not meet the criterion is \eqref{HB8}. However, in light of Theorem~\ref{th:relative-density} it suffices to find a pure formula equivalent to $\Conv p\rarrow\Conv\Conv p$, and take it as an axiom instead of \eqref{HB8}. To this end, we apply the Sahlqvist-van Benthem algorithm:\footnote{See \citep{Blackburn-et-al-ML} and \citep{tenCate-MTFEML} for details. In particular, Section 3.6 ``Sahlqvist Formulas'' of the former, and Section 3.2 ``First-order correspondence languages'' of the latter, with one caveat. Standardly, for nominals we have $ST_x(i)\defeq x = c_i$, where $c_i$ is a new individual constant. Here, as we consider all valuations for the nominals, and thus need to quantify over ``them'', we translate elements of $\Omega$ to first-order variables: $ST_x(i)\defeq x = x_i$.}

\begin{align*}
\frF\Vdash \Conv p\to \Conv\Conv p\quad&\text{iff}\quad\frF\vDash\forall P\forall x(ST_x(\Conv p)\to ST_x(\Conv\Conv p))\\
&\text{iff}\quad\frF\vDash\forall P\forall x(ST_x(\langle B\rangle(p,p))\to ST_x(\Conv\Conv p))\\
&\text{iff}\quad\frF\vDash\forall P\forall x(\exists y\exists z(B(y,x,z)\land Py\land Pz)\to ST_x(\Conv\Conv p))\\
&\text{iff}\quad\frF\vDash\forall P\forall x\forall y\forall z(B(y,x,z)\land Py\land Pz\to ST_x(\Conv\Conv p))\\
&\text{iff}\quad\frF\vDash\forall x\forall y\forall z(B(y,x,z)\to ST_x(\Conv\Conv p))[\nicefrac{\lambda u.(u=y\lor u=z)}{P}]\\
&\text{iff}\quad\frF\vDash\forall x\forall{x_i}\forall{x_j}(ST_x(\langle B\rangle(i,j))\to ST_x(\Conv\Conv(i\lor j)))\\
&\text{iff}\quad\frF\Vdash\langle B\rangle(i,j)\to \Conv\Conv(i\lor j)\,.
\end{align*}

Recall that \eqref{HB8} defines the class of dense 3-frames <<directly>>, while $\Conv p\rarrow\Conv\Conv p$ only relatively, and so does the equivalent formula
\begin{equation}\tag{$\mathtt{HB8'}$}\label{HB8'}
    \Bmod(i,j)\to \Conv\Conv(i\lor j)\,.
\end{equation}
However, we can still use it to capture a class of dense frames by means of the system with pure formulas only, as we have pure formulas that define that class $\LBWE$, the setting in which by means of Theorem~\ref{th:relative-density} formula \eqref{HB8'} corresponds to density.

\section{The hybrid logic of strict betweenness}\label{sec:hybrid-logic-of-strict-betweenness}

In this section, we propose the hybrid logic of strict betweenness for the class $\DLBWE$. The system will be a pure axiomatic extension of the basic system for the language $\Hbat$, so we get automatic completeness similar to \cite[Section 7.3]{Blackburn-et-al-ML}.

\subsection{The system \texorpdfstring{$\Bh$}{Bh}}

The hybrid logic of strict betweenness will be formulated in the language $\Hbat$. The basic axioms and rules are similar to those from \cite[Section 7.3]{Blackburn-et-al-ML} for the system $\mathbf{K}_h+\text{\textsc{Rules}}$, and  are adapted for the case of the binary modality $\Bmod$. Thus, we assume the following axioms:
\begin{gather}
    \text{All classical propositional tautologies}\tag{$\mathtt{CT}$}\\
    \neg\Bmod(p,q)\leftrightarrow\Nec(\neg p,\neg q)\tag{$\mathtt{dual}$}\\
    \Nec(p\to q,r)\to(\Nec(p,r)\to\Nec(q,r))\tag{$\mathtt{K}_1$}\\
    \Nec(r,p\to q)\to(\Nec(r,p)\to\Nec(r,q))\tag{$\mathtt{K}_2$}\\
    \at_i(p\to q)\to(\at_i p\to \at_i q)\tag{$\mathtt{K}_{\at}$}\\
    \neg\at_i p\leftrightarrow\at_i\neg p\tag{$\mathtt{selfdual}$}\\
    \at_i i\tag{$\mathtt{ref}$}\\
    i\wedge p\to \at_i p\tag{$\mathtt{intro}$}\\
    \Bmod(\at_i p,q)\to \at_i p\tag{$\mathtt{back}_1$}\\
    \Bmod(q,\at_i p)\to \at_i p\tag{$\mathtt{back}_2$}\\
    \at_i\at_j p\to \at_j p\tag{$\mathtt{agree}$}\\
    \at_ji\rightarrow\at_ij\tag{$\mathtt{sym}$}\\
    \at_ij\wedge\at_jp\rightarrow\at_ip\tag{$\mathtt{nom}$}
\end{gather}
The rules of the system are:
\begin{gather}
    \text{From $\vdash\varphi\to\psi$ and $\vdash\varphi$ infer $\vdash\psi$.}\tag{$\mathtt{MP}$}\\
    \text{From $\vdash\varphi$ infer $\vdash\Nec(\varphi,\psi)$.}\tag{$\mathtt{Nec}_1$}\\
    \text{From $\vdash\psi$ infer $\vdash\Nec(\varphi,\psi)$.}\tag{$\mathtt{Nec}_2$}\\
    \text{From $\vdash\varphi$ infer $\vdash\at_i\varphi$}\tag{$\mathtt{Nec}_{\at}$}\\
    \text{\parbox{0.5\textwidth}{From $\vdash\varphi$ infer $\vdash\sigma(\varphi)$ where $\sigma(\varphi)$ is a uniform substitution which replaces propositional variables by formulas and nominals by nominals.}}\tag{$\mathtt{Subst}$}\\[+3pt]
    \text{\parbox{0.5\textwidth}{From $\vdash i\rightarrow\theta$ infer $\vdash\theta$, if $i$ does not occur in $\theta$.}}\tag{$\mathtt{Name}$}\label{eq:Name}\\[+3pt]
    \text{\parbox{0.5\textwidth}{From $\vdash\at_i\Bmod(j,k)\wedge \at_j\varphi\wedge\at_k\psi\rightarrow\theta$ infer $\vdash\at_i\Bmod(\varphi,\psi)\rightarrow\theta$, if $i,j,k$ are pairwise different and $j,k$ do not occur in $\varphi,\psi$ and $\theta$.}}\tag{$\mathtt{Paste}$}\label{eq:Paste}
\end{gather}

\pagebreak

Additional specific axioms are \eqref{HB1}--\eqref{HB7}, \eqref{HB8'}, which have geometric meaning and which are to guarantee that the logic we obtain is the logic of countable dense linear betweenness frames without endpoints. We will denote this system by $\Bh$ (a~hybrid logic of betweenness).

\subsection{Completeness of \texorpdfstring{$\Bh$}{Bh}}

For the completeness proof, we essentially follow the proof strategy from \cite[Section 7.3]{Blackburn-et-al-ML}. We list the relevant lemmas and only prove them when it is related to the binary modality.

\begin{definition}
A model $\frM\defeq\langle W,B,V\rangle$ is \emph{named} if for all $x\in W$, there is a nominal $i$ such that $V(i)=\{x\}$. A $\Bh$-maximal consistent set $\Gamma$ is \emph{named} if it contains a nominal (and this nominal is a \emph{name for} $\Gamma$).
\end{definition}

\begin{definition}
Given a pure formula $\varphi$, we say that $\psi$ is a \emph{pure instance} of $\varphi$ if is obtained from $\varphi$ by uniformly substituting nominals for nominals.
\end{definition}

\begin{lemma}[{\citealp[Lemma 7.22]{Blackburn-et-al-ML}}]\label{Lemma:model:to:frame}
Given a named model $\frM\defeq\langle\frF,V\rangle$ and a pure formula $\varphi$, if $\frM\Vdash\psi$ for all pure instances $\psi$ of $\varphi$, then $\frF\Vdash\varphi$.
\end{lemma}

\begin{lemma}[{\citealp[Lemma 7.24]{Blackburn-et-al-ML}}]\label{Lemma:yield}
Let $\Gamma$ be a $\Bh$-maximal consistent set. For every nominal $i$, let $\Delta_i:=\{\varphi\mid\at_i\varphi\in\Gamma\}$ (a \emph{named set yielded by} $\Gamma$). Then,
\begin{enumerate}[label=(\roman*),itemsep=0pt]
\item For each nominal $i$, $\Delta_{i}$ is an $\Bh$-maximal consistent set containing $i$.
\item For all nominals $i$ and $j$, if $i\in\Delta_j$ then $\Delta_i=\Delta_j$.
\item For all nominals $i$ and $j$, $\at_i\varphi\in\Delta_j$ iff $\at_i\varphi\in\Gamma$.
\item If $k$ is a name for $\Gamma$, then $\Gamma=\Delta_k$.
\end{enumerate}
\end{lemma}

\begin{definition}
A $\Bh$-maximal consistent set $\Gamma$ is \emph{pasted} if $\at_i\Bmod(\varphi,\psi)\in\Gamma$ implies that for some nominals $j$ and $k$, $\at_i\Bmod(j,k)\land\at_j\varphi\land\at_k\psi\in\Gamma$.
\end{definition}

\enlargethispage{50pt}

\begin{lemma}[Extended Lindenbaum Lemma, after {\citealp[Lemma 7.25]{Blackburn-et-al-ML}}]
Let $\Omega'$ be a (countably) infinite collection of nominals disjoint from $\Omega$, and let $\Hbat'$ be the language obtained by adding these new nominals to $\Hbat$. Then every $\Bh$-consistent set of formulas in language $\Hbat$ can be extended to a named and pasted $\Bh$-maximal consistent set in language $\Hbat'$.
\end{lemma}
\begin{proof}[Proof sketch] We reason in an analogous way as the authors of \citep{Blackburn-et-al-ML}. Begin with an enumeration of $\Omega'$. Fix a consistent set $\Sigma$ and expand it to $\Sigma_0\defeq\Sigma\cup\{k\}$, where $k$ is the first nominal in the enumeration. Enumerate all formulas in the new language $\Hbat'$. By \eqref{eq:Name} $\Sigma_0$ is consistent. Then, assume that $\Sigma_n$ has been defined and is consistent, and pick the $n+1$-th formula in the enumeration, say $\varphi_{n+1}$. If $\Sigma_n\cup\{\varphi_{n+1}\}$ is inconsistent, then $\Sigma_{n+1}\defeq\Sigma_n$. Otherwise, there are two possibilities:
\begin{enumerate}[label=(\roman*),itemsep=0pt]
    \item $\varphi_{n+1}$ has the form different from $\at_i\Bmod(\varphi,\psi)$, in which case we put $\Sigma_{n+1}\defeq\Sigma_n\cup\{\varphi_{n+1}\}$;
    \item $\varphi_{n+1}$ is of the form $\at_i\Bmod(\varphi,\psi)$, in which case we put
    \[
    \Sigma_{n+1}\defeq\Sigma_n\cup\{\varphi_{n+1}\}\cup\{\at_i\Bmod(j,k)\land\at_j\varphi\land\at_k\psi\}
    \]
    where $j$ and $k$ are among the new nominals (as we are at the finite stage of construction we will always have new nominals); \eqref{eq:Paste} rule guarantees that the transition from $\Sigma_n$ to $\Sigma_{n+1}$ preserves consistency.
\end{enumerate}
    We put $\Gamma\defeq\bigcup_{n=0}^{\infty}\Sigma_n$. The set is named, pasted, and $\Bh$-maximal consistent.
\end{proof}

\begin{definition}[{\citealp[Definition 7.26]{Blackburn-et-al-ML}}]
Let $\Gamma$ be a named and pasted $\Bh$-maximal consistent set. The \emph{named model yielded by} $\Gamma$ is the triple $\frM^{\Gamma}\defeq\langle W^{\Gamma},B^{\Gamma},V^{\Gamma}\rangle$, where
\begin{enumerate}[label=(\roman*),itemsep=0pt]
\item $W^{\Gamma}$ is the set of all named sets
yielded by $\Gamma$.
\item $B^{\Gamma}(u,v,w)$ iff $\Bmod(\varphi,\psi)\in v$ whenever $\varphi\in u$ and $\psi\in w$.
\item $V^{\Gamma}(p)\defeq\{x\in W^{\Gamma}\mid p\in x\}$ and $V^{\Gamma}(i)\defeq\{x\in W^{\Gamma}\mid i\in x\}$.
\end{enumerate}
\end{definition}

By the first two items of Lemma \ref{Lemma:yield}, we see that for every nominal $i$, $V^{\Gamma}(i)=\{\Delta_i\}$ (and so $V^\Gamma$ is correctly defined, as its values for the nominals are singleton subsets of the universe).

For the Existence Lemma below we need the fact that the formula
\begin{equation}\label{eq:bridge}\tag{$\mathtt{bridge}$}
\Bmod(j,k)\land \at_j\varphi\land\at_k\psi\to\Bmod(\varphi,\psi)\footnotemark
\end{equation}
\footnotetext{This is a binary counterpart of the formula (Bridge) used by \cite{Blackburn-et-al-ML}.}
is a theorem of $\Bh$. We leave the proof to the reader.

\begin{lemma}[{Existence Lemma, \citealp[Lemma 7.27]{Blackburn-et-al-ML}}]\label{lem:existence}
Let $\Gamma$ be a named and pasted $\Bh$-maximal consistent set, and let $\frM^{\Gamma}=\langle W^{\Gamma},B^{\Gamma},V^{\Gamma}\rangle$ be the named model yielded by $\Gamma$. If $u\in W^\Gamma$ and $\Bmod(\varphi,\psi)\in u$, then there are $v,w\in W^\Gamma$ such that $B^\Gamma(v,u,w)$ and $\varphi\in v$, $\psi\in w$.
\end{lemma}
\begin{proof}
    Again, the proof is basically the same as that for the unary case. Let $u\in W^{\Gamma}$ be such that $\Bmod(\varphi,\psi)\in u$. Then, since the model is named, there is an $i$ such that $u=\Delta_i$. Thus $\at_i\Bmod(\varphi,\psi)\in\Gamma$. As $\Gamma$ is pasted, there are nominals $j$ and $k$ such that $\at_i\Bmod(j,k)\wedge\at_j\varphi\land\at_k\psi\in\Gamma$. In consequence, $\varphi\in\Delta_j$ and $\psi\in\Delta_k$.

    Take $\delta\in\Delta_j$ and $\gamma\in\Delta_k$. Thus, $\at_j\delta\wedge\at_k\gamma\in\Gamma$. By Lemma~\ref{lem:4-properties}(iii) $\at_j\delta\wedge\at_k\gamma\in\Delta_i$. As $\at_i\Bmod(j,k)$ is an element of $\Delta_i$, by \eqref{eq:bridge} we get that $\Bmod(\delta,\gamma)\in\Delta_i$. Thus $B^{\Gamma}(\Delta_j,\Delta_i,\Delta_k)$, as required.
\end{proof}

\begin{lemma}[{Truth Lemma, \citealp[Lemma 7.28]{Blackburn-et-al-ML}}]
Let $\frM^{\Gamma}=\langle W^{\Gamma},B^{\Gamma},V^{\Gamma}\rangle$ be the named model yielded by a named and pasted $\Bh$-maximal consistent set $\Gamma$, and let $u\in W^\Gamma$. Then for all formulas $\varphi$, we have $\varphi\in u$ iff $\frM, u\Vdash\varphi$.
\end{lemma}

\begin{theorem}[{Completeness Theorem, \citealp[Lemma 7.29]{Blackburn-et-al-ML}}]\label{th:completeness-for-Bh}
Every $\Bh$-consistent set of formulas is satisfiable in a countable named model based on a frame that
validates every extra betweenness axiom of logic $\Bh$. Hence every $\Bh$-consistent set of formulas is satisfiable on a model based on a countable frame from the class $\DLBWE$.
\end{theorem}

\begin{proof}
Given an $\Bh$-consistent set of formulas $\Sigma$, use the Extended Lindenbaum Lemma to expand it to a named and pasted set $\Gamma$ in a countable language $\Hbat'$. Let $\frM^\Gamma\defeq\langle W^\Gamma,B^\Gamma,V^\Gamma\rangle$ be the named model yielded by $\Gamma$. Since $\Gamma$ is named, it is in $W^\Gamma$. By the Truth
Lemma, for every $\varphi\in\Sigma$ we have that $\frM^\Gamma,\Gamma\Vdash\varphi$. The model is countable because each state is named by some nominal in $\calL'_B$, and there are only countably many of these.

To see that the canonical frame $\frF^\Gamma=\langle W^\Gamma,B^\Gamma\rangle$ validates all eight specific betweenness axioms, observe that all these belong to every maximal consistent set in $W^\Gamma$. Therefore, they are globally true in $\frM^\Gamma$. Hence, by Lemma \ref{Lemma:model:to:frame}, $\frF^\Gamma$ validates all those axioms too.
\end{proof}

\begin{corollary}
    Every $\Hbat$-formula valid on the frame $\frQ=\langle\Rat,B_<\rangle$ is a~theorem of $\Bh$.
\end{corollary}
\begin{proof}
    Suppose $\nvdash_{\Bh}\varphi$, which means that $\{\neg\varphi\}$ is $\Bh$-consistent. Thus by Theorem~\ref{th:completeness-for-Bh} there is a model $\frM\defeq\langle\frF,V\rangle$ such that $\frF\in\DLBWE$ is countable and $\frM,w\Vdash\neg\varphi$, for some world $w$. Yet being countable, $\frF$ is isomorphic to $\frQ$, so $\varphi$ fails on a model based on $\frQ$. In consequence, any formula valid on $\frQ$ is a theorem of $\Bh$.
\end{proof}

\subsection{The system \texorpdfstring{$\Bhp$}{Bh+}}

We now move on to the subclass of $\DLBWE$ that contains only complete structures, i.e., those which satisfy the following second-order axiom
\begin{equation}\tag{$\mathtt{B9}$}\label{B9}
\begin{split}
    (\forall X,Y\in 2^W)\,[(\exists w\in W)&(\forall x\in X)(\forall y\in Y)\,B(w,x,y)\rarrow\\
    &(\exists u\in W)(\forall x\in X)(\forall y\in Y)\,B(x,u,y)]\,.
    \end{split}
\end{equation}
We define
\[
\CDLBWE\defeq\DLBWE+\eqref{B9}\,,
\]
the class of \emph{Dedekind complete} (\emph{D-complete}, for short)\footnote{We use the term to avoid confusion with logical completeness.} dense linear betweenness structures without endpoints.

Being second-order, \eqref{B9} axiom cannot have a pure counterpart formula built in $\Hbatex$. Moreover, since all elements of the class are uncountable, $\CDLBWE$ cannot be axiomatized by means of first-order formulas (by L\"owenheim-Skolem theorem). In consequence,
\begin{theorem}
    The class $\CDLBWE$ is not $\Hbatex$-definable by pure formulas.
\end{theorem}

However, we can express D-completeness by means of a formula containing the $\Ex$ operator and propositional variables. Precisely, as the axioms of $\Bh$ locate us in a subclass of $\LBWE$, Theorem \ref{th:E-elimination-in-LBWE} allows us to express this formula in a definitional extension of $\Lb$.  Therefore, the axiomatic system which results from adding the axiom \eqref{ax:D} below to $\Bh$, has the operator $\Ex$ as an abbreviation couched in the language $\calL_{\beta}$ (the more so in $\Hbat$). Let $\All\defeq\neg\Ex\neg$. The modal Dedekind completeness axiom is the following formula
\begin{equation}\tag{$\mathtt{D}$}\label{ax:D}
\begin{split}
\Ex\Conv p\land\Ex\Conv q\land \All(\Conv p\to p)\land\All(\Conv q\to q&)\land \neg\Ex(p\land q)\rarrow\\&\Ex(\Bmod(p,q)\land\neg\Conv p\land \neg\Conv q)\,.
\end{split}
\end{equation}
The axiom embodies the following intuition. Suppose there exist a pair of distinct points satisfying $p$ and a pair of distinct points satisfying $q$ (this is postulated by $\Ex\Conv p\land\Ex\Conv q$). Since all points between any pair of $p$-points satisfy $\Conv p$, $\All(\Conv p\rarrow p)$ postulates that the set of all $p$-points is an interval (i.e., a convex set). Analogously, the set of all $q$-points is an interval too. By $\neg\Ex(p\land q)$ the two intervals are disjoint. Shortly, the antecedent of \eqref{ax:D} says: there is a pair of non-empty disjoint intervals (thus one of them being entirely to the one side of another). Then, there is a point that is located between a $p$-point and a $q$-point, but not between any pair of $p$-points, nor any pair of $q$-points. Briefly, there is a point between the two intervals from the antecedent.

\pagebreak

\begin{lemma}
    \eqref{ax:D} is valid in the class $\CDLBWE$.
\end{lemma}
\begin{proof}
    Take an arbitrary $\frF\in\CDLBWE$, a~point $w\in\frF$, and a valuation $V$. Suppose that
    \[
        w\Vdash \Ex\Conv p\land\Ex\Conv q\land \All(\Conv p\to p)\land\All(\Conv q\to q)\land \neg\Ex(p\land q)\,.
    \]
    Let $x$ and $y$ be such that $x\Vdash\Conv p$ and $y\Vdash\Conv q$. In consequence, there are $x_1,x_2$ and $y_1,y_2$ such that
    \[
            B(x_1,x,x_2)\tand \{x_1,x_2\}\subseteq V(p) \qquad B(y_1,y,y_2)\tand \{y_1,y_2\}\subseteq V(q)\,.
    \]
    By the assumptions, we have that ($\dagger$) $V(p)\cap V(q)=\emptyset$ and that both sets are convex. That is, if $a,b\in V(p)$, then for any $u$ such that $B(a,u,b)$ we obtain $u\in V(p)$ (and similarly for $V(q)$). Indeed, if $\{a,b\}\subseteq V(p)$ and $u$ is between the two points, then $u\Vdash\Conv p$, and so by $\All(\Conv p\rarrow p)$ we get that $u\in V(p)$.

    Take $x_0\in V(p)$ such that $x$ is between $x_0$ and an element of (equivalently, all elements of) $V(q)$, and $y_0\in V(q)$ such that $y$ is between $y_0$ and en element of (again, equivalently, all elements of) $V(p)$. To be more precise, $x_0$ can be one of $x_1$ and $x_2$, and $y_0$ one of $y_1$ and $y_2$.

    To see the former, suppose $v\in V(q)$. Then from $V(p)\cap V(q)=\emptyset$ we get $\#(x,x_1,v)$ and $\#(x,x_2,v)$. For $x_1$, from \eqref{B6} we have $B(x_1,x,v)$ or $B(x_1,v,x)$ or $B(x,x_1,v)$. Since $x\Vdash p$ by convexity of $V(p)$, then from $B(x_1,v,x)$ we obtain $v\Vdash p$, which contradicts ($\dagger$). If $B(x_1,x,v)$, then we take $x_0\defeq x_1$. If $B(x,x_1,v)$, then from $B(x_1,x,x_2)$ and \eqref{B2} and \eqref{B4} we get $B(x_2,x,v)$. So we put $x_0\defeq x_2$. For $y_0$ the argument is similar.

    Let
    \begin{align*}
        X_0&{}\defeq V(p)\cap\{u\in W\mid B(x_0,u,y_0)\}\\
        Y_0&{}\defeq V(q)\cap\{u\in W\mid B(x_0,u,y_0)\}\,.
    \end{align*}
    The choice of $x_0$ and $y_0$ guarantees that both sets are non-empty (infinite, actually, in light of density and convexity). Since $x_0\neq y_0$, Lemma \ref{lem:exists-to-one-side} entails the existence of $z\in W$ such that $B(z,x_0,y_0)$. Fix arbitrary $a\in X_0$ and  $b\in Y_0$. Since $\#(z,a,b)$, we obtain
\[
B(z,a,b)\vee B(z,b,a)\vee B(a,z,b)
\]
by linearity.

If $B(a,z,b)$, then from $B(x_0,a,y_0)$ and $B(x_0,b,y_0)$ and Lemma \ref{lem:for-D-soundness} we get $B(x_0,z,y_0)$, a contradiction to $B(z,x_0,y_0)$ and \eqref{B3}.

Suppose that $B(z,b,a)$. From  $B(z,x_0,y_0)$, $B(x_0,a,y_0)$, and \eqref{B5} we obtain $B(z,x_0,a)$. Since $b\notin V(p)$, by the convexity of $V(p)$ and $x_0,a\in V(p)$ we have that $B(x_0,b,a)$ cannot be true, and since $\#(x_0,b,a)$, we have that either $B(b,x_0,a)$ or $B(b,a,x_0)$ from \eqref{B6} and \eqref{B2}. From $B(z,b,a)$ via application of \eqref{B4} and \eqref{B5} in both cases we have $B(z,b,x_0)$. Then from $B(z,x_0,y_0)$ and \eqref{B2}, \eqref{B5} we have $B(b,x_0,y_0)$, which contradicts $B(x_0,b,y_0)$ by \eqref{B2} and \eqref{B3}.

Therefore the third possibility holds, and by arbitrariness of $a$ and $b$ we conclude that $(\forall a\in X_0)(\forall b\in Y_0)\,B(z,a,b)$. By \eqref{B9}, there is a $z_0\in W$ such that $(\forall a\in X_0)(\forall b\in Y_0)\,B(a,z_0,b)$. Thus, $z_0\Vdash\Bmod(p,q)$, and by irreflexivity of $B$, $z_0\notin X_0\cup Y_0$. Moreover, observe that
     \[
        B(x_0,x,y_0)\wedge B(x_0,y,y_0)\wedge B(x,z_0,y)
     \]
     so by Lemma~\ref{lem:for-D-soundness} we obtain $B(x_0,z_0,y_0)$. In consequence $z_0\notin V(p)\cup V(q)$, and therefore $z_0\notin V(\Conv p)\cup V(\Conv q)$.
\end{proof}

\subsection{Completeness of \texorpdfstring{$\Bhp$}{Bhp}}

For a fixed $\Bhp$-consistent set of sentences $\Gamma$, the canonical model $\frM^\Gamma$ yielded by it satisfies all the axioms of $\Bh$, and thus the canonical frame $\frF^\Gamma$---on which $\frM^\Gamma$ is based---validates the axioms by Lemma~\ref{Lemma:model:to:frame}. Since the frame being named is countable, it is isomorphic to $\frQ=\langle\Rat,B_<\rangle$. \eqref{ax:D} is of course globally true in $\frM^\Gamma$, yet it cannot be valid on $\frF^\Gamma$, as otherwise $\frF^\Gamma$, and so $\frQ$, would be D-complete.

Thus, to show completeness of $\Bhp$ with respect to $\CDLBWE$ we need to expand the domain of $\frF^\Gamma$ with new elements that will guarantee D-completeness. However, the expansion must be made in such a way that betweenness holds among new points, and the truth lemma is preserved for the new model. So in particular, the truth lemma for the old ``rational'' points is not affected by the expansion.

Due to the isomorphism between $\frF^\Gamma$ and $\frQ$, we can tag the points of the former frame with rational numbers: $\{\Gamma_a\mid a\in\Rat\}$ in such a way that $B^\Gamma(\Gamma_a,\Gamma_b,\Gamma_c)$ if and only if $B_{<}(a,b,c)$. In the event, for the sake of simplicity below we replace $\frF^\Gamma$ with $\frQ$.

Unless declared otherwise, we will use the initial letters of the Latin alphabet: $a$, $b$, $c$ and $d$ (indexed, if necessary) to range over the set of rationals. The letters from the end of the Latin alphabet, $x$, $y$, and $z$, will be used to range over the whole set of real numbers.

If $x$ is an irrational number and $A\subseteq\Rat$, then we write $x\lessdot A$ to indicate that $x$ is an element of the minimal completion of $A$ (i.e., the subset of the reals that completes $A$). When we mention an interval as a subset of $\Rat$, we mean the set of rational numbers in the interval.

\begin{definition} Given a rational based model $\frM\defeq\langle \frQ,V\rangle$ and an irrational number $x$, we call a formula $\phi$ \emph{local at} $x$, if there are $a,a'\in\Rat$ such that $x\lessdot (a,a')$ and $(a,a')\subseteq V(\phi)$.
\end{definition}
The idea is that $\phi$ is local at $x$, if there is an interval $(a,a')\subseteq\Rat$ such that $\phi$ is satisfied by every its point. In principle, we want that if all the rational points that are in the vicinity of $x$ make $\phi$ true, then $\phi$ should also be true at~$x$.

For a fixed irrational number $x$, let us define the following three sets
\begin{align*}
\Lambda_x&{}\defeq\{\phi \mid \phi \mbox{ is local at }x\}\\
\Sigma_x&{}\defeq\{\langle B\rangle(\phi,\psi)\mid\mbox{there are }a,a'\mbox{ such that }a\Vdash \phi, a'\Vdash\psi\mbox{ and }x\lessdot(a,a')\}\\
\Pi_x&{}\defeq\{\neg\langle B\rangle(\phi,\psi)\mid \langle B\rangle(\phi,\psi)\notin\Sigma_x\}\,.
\end{align*}

\begin{definition}
    The logic $\Bhp$ is \emph{sound} with respect to a model $\frM$, if all theorems of $\Bhp$ are globally true in $\frM$.
\end{definition}

\begin{lemma}\label{lem:4-properties} Let $\frM\defeq\langle \frQ,V\rangle$ be a rational based model. Then,
    \begin{enumerate}[label=(\roman*),itemsep=0pt]
       \item For every $x$, $\Sigma_x\subseteq\Lambda_x$.
       \item For every $x$, $\Lambda_x$ is closed for conjunctions.
       \item For each nominal $i$, $\neg i$ is local.
       \item If $\Bhp$ is sound w.r.t. $\frM$, then for every $x$ and every $\phi\in\Lambda_x$, $\nvdash_{\Bhp}\phi\rarrow\bot$, so $\Lambda_x$ is $\Bhp$-consistent.
    \end{enumerate}
\end{lemma}
\begin{proof}
    Ad i. If $a\Vdash\phi$ and $a'\Vdash\psi$, then all the points between $a$ and $a'$ must satisfy $\Bmod(\phi,\psi)$. Therefore if $x\lessdot(a,a')$, then $\Bmod(\phi,\psi)$ is local at $x$.

    \smallskip

    Ad ii. Suppose both $\phi$ and $\varphi$ are local at $x$, in which case there are rational numbers $a,a',b,b'$ such that $x\lessdot(a,a')\subseteq V(\phi)$ and $x\lessdot(b,b')\subseteq V(\varphi)$. Then
    \[
    x\lessdot (a,a')\cap (b,b')\subseteq V(\phi)\cap V(\varphi)=V(\phi\wedge\varphi)
\,,    \]
    and so $\phi\wedge\varphi$ must be local at $x$.

    \smallskip

    Ad iii. This follows from the fact that each nominal is true at precisely one rational point. So, given $i$, suppose that $V(i)=\{a\}$. Thus for any irrational number $x$, as either $a<x$ or $x<a$, there is an interval $(b,b')$ around $x$ at every point of which $\neg i$ holds.

    \smallskip

    Ad iv. Suppose there are a point $x$ and a formula $\phi$ local at $x$ such that $\vdash_{\Bhp}\phi\rarrow\bot$. By locality there are rationals $a<a'$ such that $x\lessdot(a,a')\subseteq V(\phi)$. Pick a rational $b\in(a,a')$. By soundness, $b\in V(\phi\rarrow\bot)$, and thus $b\in V(\bot)$, a~contradiction.
\end{proof}

\begin{lemma}\label{lem:CCp->Cp}
    The formula $\Conv\Conv p\rarrow\Conv p$ is valid in $\LBWE$.
\end{lemma}
\begin{proof}
Fix an $\frF\in\LBWE$, an arbitrary point $x$ of the frame, and an arbitrary valuation $V$. Suppose $\langle\frF,V\rangle,x\Vdash\Conv\Conv p$. Then there are points $y$ and $z$ such that $B(y,x,z)$ and both satisfy $\Conv p$. So, there are points $y_0,y_1$ and $z_0,z_1$ such $B(y_0,y,y_1)$ and $B(z_0,z,z_1)$ and all four points satisfy $p$. In consequence, by Lemma~\ref{lem:for-CCp->Cp}, $x$ is between a pair of points that satisfy $p$, and thus $\langle\frF,V\rangle,x\Vdash\Conv p$, as required.
\end{proof}

\begin{lemma}\label{lem:for-completeness-of-R}
    Fix a rational based model $\frM\defeq\langle\frQ,V\rangle$ with respect to which $\Bhp$ is sound. If there is a natural number $n$ such that the formula
    \[
\Bmod(\phi_1,\psi_1)\vee\ldots\vee \Bmod(\phi_n,\psi_n)
    \]
    is local at $x$, then there exists $i_x\leqslant n$ such that $\Bmod(\phi_{i_x},\psi_{i_x})\in\Sigma_x$.
\end{lemma}

\begin{proof}
    We prove the lemma via induction on $n$. For $n=1$, let us assume that $x\lessdot(a,a')$ and $(a,a')\subseteq V\left(\Bmod(\phi,\psi)\right)$. Fix a rational $b\in (a,a')$. We have $b\Vdash\Bmod(\phi,\psi)$. Let $b_1$ be such that $b_1\Vdash\phi$, and w.l.o.g. assume that $b_1<x$. If there is a $b_2$ such that $x<b_2$ and $b_2\Vdash\psi$, then $\Bmod(\phi,\psi)\in\Sigma_x$. Otherwise, for all $c$ such that $x<c$ it is the case that $c\nVdash\psi$. But ($\ddagger$) $(x,a')\subseteq V(\Bmod(\phi,\psi))$, so there must be a $d>x$ such that $d\Vdash\phi$. Thus there are points on both sides of $x$ that satisfy $\phi$. However, ($\ddagger$) entails the existence of a rational that satisfies $\psi$ and this rational is either to the left or to the right of $x$. In consequence, there are rationals between which $x$ is located and such that one of them satisfies $\phi$ and the other $\psi$. Thus $\Bmod(\phi,\psi)\in\Sigma_x$, as required.

      Suppose now that the result holds for $n$ and assume that there are rationals $a$ and $a'$ such that
    \[
    x\lessdot[a,a']\subseteq V\left(\bigvee_{i\leqslant n+1}\Bmod(\phi_i,\psi_i)\right)
    \]
    (for convenience, we take closed intervals instead of open ones, and we can always do this due to the density of $\Rat$ in $\Real$).

    Thus, for $a$ there is $i_a\leqslant n+1$ such that $a\Vdash\Bmod(\phi_{i_a},\psi_{i_a})$. So there are rational numbers $a_1$ and $a_2$ such that $a_1\Vdash\phi_{i_a}$ and $a_2\Vdash\psi_{i_a}$ and $a$ is between the two, say $a_1<a<a_2$. Clearly, $(a_1,a_2)\subseteq V(\Bmod(\phi_{i_a},\psi_{i_a}))$.

    Clearly, $a_1<x$. If $x<a_2$, then $x\in(a_1,a_2)$, and $\Bmod(\phi_{i_a},\psi_{i_a})\in\Sigma_x$. So consider the case in which $a_2<x$. Let us take the set $L\defeq\{b\mid b\Vdash\phi_{i_a}\vee\psi_{i_a}\}$.

    If $x$ is not its upper bound, then there is a $b_1>x$ satisfying $\phi_{i_a}\vee\psi_{i_a}$. If $b_1\Vdash\psi_{i_a}$, then $x$ is between rational points satisfying $\phi_{i_a}$ and $\psi_{i_a}$, respectively, so $\Bmod(\phi_{i_a},\psi_{i_a})\in\Sigma_x$. If $b_1\Vdash\phi_{i_a}$, then since $a_2<x$ and $a_2$ satisfies $\psi_{i_a}$ we are in the same situation.

    If $x$ is an upper bound of $L$, then by completeness of the real line, there is an $s\defeq\sup L$. Consider the case when $s<x$. Then, there is an $a_s\in\Rat$ such that $s<a_s<x$. Observe that $(a_s,a')\cap V(\Bmod(\psi_{i_a},\phi_{i_a}))=\emptyset$, since no point to the right of $a_s$ can satisfy $\phi_{i_a}$ or $\psi_{i_a}$. Then
    \[
        x\lessdot(a_s,a')\subseteq V\left(\bigvee_{\substack{i\leqslant n+1\\ i\neq i_a}}\Bmod(\phi_i,\psi_i)\right)
    \]
    By the induction hypothesis, there is an $i_x\leqslant n+1$ and $i_x\neq i_a$ such that $\Bmod(\phi_{i_x},\psi_{i_x})\in\Sigma_x$.

    The other case is when $s=x$. Firstly, we look at $a'$ for which there is $i_{a'}$ such that $a'\Vdash V(\Bmod(\psi_{i_{a'}},\phi_{i_{a'}}))$. We proceed analogously as in the case of $a$ this time putting in focus the set $U\defeq\{c\mid c\Vdash \psi_{i_{a'}}\vee\phi_{i_{a'}}\}$. When $x$ is not its lower bound, we argue that $\Bmod(\psi_{i_{a'}},\phi_{i_{a'}})\in\Sigma_x$ similarly to the situation of $L$. Otherwise, $x$ is its lower bound, the set has the infimum $t\defeq\inf U$. If $x<t$, then we proceed as in the previous paragraph.

    If $t=x$, then we get that $x$ is both the supremum of $L$ and the infimum of~$U$. As $L$ and $U$ are sets of rationals, $L\cap U=\emptyset$, and their convex hulls must be disjoint too. Consider now two formulas, $\Conv(\phi_{i_a}\vee\psi_{i_a})$ and $\Conv(\phi_{i_{a'}}\vee\psi_{i_{a'}})$. The extension of the former is the convex hull of $L$, and of the latter the convex hull of $U$. Thus $\neg\Ex(\Conv(\phi_{i_a}\vee\psi_{i_a})\wedge\Conv(\phi_{i_{a'}}\vee\psi_{i_{a'}}))$. Clearly, $\Ex\Conv(\phi_{i_a}\vee\psi_{i_a})$ and $\Ex\Conv(\phi_{i_{a'}}\vee\psi_{i_{a'}})$. By Lemma~\ref{lem:CCp->Cp}, $\Conv\Conv p\rarrow\Conv p$ is valid in $\langle\Rat,B_{<}\rangle$, and so $\All(\Conv\Conv(\phi_{i_a}\vee\psi_{i_a})\rarrow \Conv(\phi_{i_a}\vee\psi_{i_a}))$ and $\All(\Conv\Conv(\phi_{i_{a'}}\vee\psi_{i_{a'}})\rarrow\Conv(\phi_{i_{a'}}\vee\psi_{i_{a'}}))$. By the soundness assumption and by \eqref{ax:D}, there is a rational number $a^\ast$ such that
    \[
    a^\ast\Vdash\Bmod(\Conv(\phi_{i_a}\vee\psi_{i_a}),\Conv(\phi_{i_{a'}}\vee\psi_{i_{a'}}))\wedge\neg \Conv\Conv(\phi_{i_a}\vee\psi_{i_a})\wedge\neg\Conv\Conv(\phi_{i_{a'}}\vee\psi_{i_{a'}})\,.
    \]
    Consider the case in which $a^\ast<x$. The point $a^\ast$ is between points $a^\ast_1$ and $a^\ast_2$ that satisfy $\Conv(\phi_{i_a}\vee\psi_{i_a})$ and $\Conv(\phi_{i_a'}\vee\psi_{i_a'})$, respectively. Since $a^\ast_2$ is to the right of $x$, $a^\ast_1<a^\ast$. But $x$ is the supremum, so there must be $a^\ast_3$ such that $a^\ast<a^\ast_3<x$ and $a^\ast_3\Vdash\Conv(\phi_{i_a}\vee\psi_{i_a})$. This entails that $a^\ast\Vdash \Conv\Conv(\phi_{i_a}\vee\psi_{i_a})$, which is a contradiction. The case when $x<a^\ast$ is similar.
\end{proof}

\begin{corollary}
    If $\frM\defeq\langle \frQ, V\rangle$ is a model with respect to which $\Bhp$ is sound and $x$ is an irrational number, then there are no $\varphi,\phi_1,\psi_1,\ldots,\phi_n,\psi_n$ such that $\varphi\in\Lambda_x$ and $\neg\Bmod(\phi_i,\psi_i)\in\Pi_x$ for all $i\leqslant n$ and $\varphi\rarrow\bigvee_{i\leqslant n}\Bmod(\phi_i,\psi_i)$ is a theorem of $\Bhp$. In consequence, $\Lambda_x\cup\Pi_x$ is $\Bhp$-consistent.
\end{corollary}
\begin{proof}
    Suppose there are such formulas for which $\vdash_{\Bhp}\varphi\rarrow\bigvee_{i\leqslant n}\Bmod(\phi_i,\psi_i)$. By the locality of $\varphi$ at $x$ and by the soundness, there are rational numbers $a$ and $a'$ such that $x\lessdot (a,a')\subseteq V(\varphi)\subseteq V\left(\bigvee_{i\leqslant n}\Bmod(\phi_i,\psi_i)\right)$. By Lemma~\ref{lem:for-completeness-of-R}, there is an $i_x\leqslant n$ for which $\Bmod(\phi_{i_x},\psi_{i_x})$ is in $\Sigma_x$. This entails that $\neg\Bmod(\phi_{i_x},\psi_{i_x})\notin\Pi_x$, a~contradiction.
\end{proof}

\begin{remark}
  By the corollary, for any irrational number $x$, $\Lambda_x\cup\Pi_x$ has
  a~maximal $\Bhp$-consistent extension $\Gamma_x$ (which is not named by Lemma~\ref{lem:4-properties}(iii), and we do not require it to be pasted either). We want to use $\Gamma_x$'s to expand the frame $\frFG$ in such a way as to obtain a frame that is isomorphic to $\frR\defeq\langle\Real,B_{<}\rangle$. Since it may happen that for different irrationals $x$ and $y$, the sets $\Gamma_x$ and $\Gamma_y$ coincide, we think about the irrational numbers as placeholders to store these sets. To be formally precise, the new points of the frame could be pairs $\langle\Gamma_x,x\rangle$. To avoid the notation cluttering we avoid this, and we always think about $\Gamma_x$ and $\Gamma_y$ as different points in the case $x\neq y$.
  \end{remark}

Let $\Gamma$ be a $\Bhp$-consistent set of sentences. The named model $\frM^{\Gamma}$ yielded by $\Gamma$ is countable, and all axioms of $\Bh$ are globally true in it. Therefore, $\frF^{\Gamma}$ is isomorphic to $\frQ$.  We expand $\frF^\Gamma$ to the frame $\frF^\Gamma_+$ such that
\begin{enumerate}[label=(\roman*),itemsep=0pt]
    \item the domain $W^{\Gamma}_{+}$ of $\frF^\Gamma_+$ is $W^\Gamma$ extended with all the $\Gamma_x$'s where $x$ is an irrational number,
    \item $\BGp(\Gamma_x,\Gamma_y,\Gamma_z)\Iffdef x<y<z\text{\ or\ }z<x<y$ (for all reals $x,y,z$),
\end{enumerate}
and we turn the frame into a model $\frM^{\Gamma}_{+}\defeq\langle\frFGp,\VGp\rangle$ in the standard way,
\[
\Gamma_x\in\VGp(p)\Iffdef p\in\Gamma_x\,.
\]
For nominals, $\VGp(i)=\VG(i)$, as these by Lemma~\ref{lem:4-properties}(iii) can only be true at rational indexed sets:
\[
\Gamma_x\in \VGp(i)\Iff \Gamma_x\in\VG(i)\Iff x\in\Rat\text{\ and\ } i\in\Gamma_x\,.
\]
It is easy to see that the restriction of $\frMGp$ to $\frFG$ is the same as the rational canonical model $\frMG$ yielded by $\Gamma$.

\begin{lemma}[Truth Lemma for $\frMGp$]\label{lem:truth-lemma+}
    For any formula $\phi$ and any number $x\in\Real$
    \[
        \frMGp,\Gamma_x\Vdash\phi\quad\text{iff}\quad\phi\in\Gamma_x\,.
    \]
\end{lemma}
\begin{proof}
    The basic cases for propositional variables and for nominals hold by the definitions and remarks above, and the Boolean connectives are handled in the standard way.

    Consider the formula $\phi\defeq\Bmod(\varphi,\psi)$, and suppose $\frMGp,\Gamma_x\Vdash\phi$. We divide the argument into two cases.

    \smallskip

    (1st case) Let $x$ be a rational number. Let $\Gamma_y$ and $\Gamma_z$ be points such that $\BGp(\Gamma_y,\Gamma_x,\Gamma_z)$, $\Gamma_y\Vdash\varphi$ and $\Gamma_z\Vdash\psi$. By the induction hypothesis $\varphi\in\Gamma_y$ and $\psi\in\Gamma_z$. W.l.o.g., assume that $y<x<z$.

    Observe that there must exist a rational number $a<x$ with $\varphi\in\Gamma_a$. For suppose there is no such rational, which in particular means that $y\notin\Rat$. Therefore, for a sufficiently small interval of rationals around $y$, all its points $a$ are such that $\neg\varphi\in\Gamma_a$. But then, $\neg\varphi$ is local at  $y$, and so $\neg\varphi\in\Gamma_y$, which is a contradiction. Analogously, there must be a rational number $b>x$ with $\psi\in\Gamma_b$. This means that $B^\Gamma(\Gamma_a,\Gamma_x,\Gamma_b)$, and so $\Bmod(\varphi,\psi)\in\Gamma_x$.

    For the other direction, suppose that $\Bmod(\varphi,\psi)\in\Gamma_x$. As $\Gamma_x$ is a point in $\frMG$, by Lemma~\ref{lem:existence} there are rationals $a$ and $b$ such that $\varphi\in\Gamma_a$ and $\psi\in\Gamma_b$ and $B^\Gamma(\Gamma_a,\Gamma_x,\Gamma_b)$. But these also hold in $\frMGp$, so by the induction hypothesis we obtain $\frMGp,\Gamma_a\Vdash\varphi$ and $\frMGp,\Gamma_b\Vdash\psi$ and $\BGp(\Gamma_a,\Gamma_x,\Gamma_b)$. Therefore $\frMGp,\Gamma_x\Vdash\Bmod(\varphi,\psi)$.

    \smallskip

    (2nd case) Let $x$ be an irrational number. If $\Gamma_x\Vdash\langle B\rangle(\varphi,\psi)$, then by the same argument as in the rational case, we can find rational numbers $a<x<b$ such that $\varphi\in\Gamma_{a}$ and $\psi\in\Gamma_{b}$. By the definition of $\Sigma_x$, we have that $\Bmod(\varphi,\psi)\in\Sigma_x\subseteq\Gamma_x$.

    Now assume $\Bmod(\varphi,\psi)\in\Gamma_x$. Since $\Sigma_x\cup\Pi_x$ contains all formulas of the form $\Bmod(\delta_1,\delta_2)$, and since $\Sigma_x\cup\Pi_x\subseteq\Gamma_x$, it must be the case that $\Bmod(\varphi,\psi)$ is an element of $\Sigma_x$ (if not, its negation would be in $\Pi_x$, and so in $\Gamma_x$, which contradicts the consistency of the latter set). In consequence, by the definition of $\Sigma_x$, there are $\Gamma_a$ and $\Gamma_b$ such that $\BGp(\Gamma_a,\Gamma_x,\Gamma_b)$ and $\frMGp,\Gamma_a\Vdash\varphi$ and $\frMGp,\Gamma_b\Vdash\psi$.
    Therefore, $\frMGp,\Gamma_x\Vdash\Bmod(\varphi,\psi)$.

    \smallskip

    This completes the proof the lemma.
    \end{proof}

\begin{theorem}[Completeness Theorem for $\Bhp$]\label{th:completeness-for-Bhp}
    If $\Gamma$ is a $\Bhp$-consistent set of formulas, then there is a model of $\Gamma$ that is based on a frame $\frF\in\CDLBWE$.
\end{theorem}
\begin{proof}
If $\Gamma$ is a $\Bhp$-consistent set, then its named and pasted maximal extension $\Gamma_a$ must correspond to a~rational number $a$, since there is a nominal $i\in\Gamma_a$, and by Lemma~\ref{lem:4-properties}(iii) no nominal is local at any irrational number. $\frMG$ is based on the frame $\frFG$ from the class $\DLBWE$. We have that $\frM^\Gamma,\Gamma_a\Vdash\Gamma$. Observe that axiom \eqref{ax:D} is globally true in $\frMG$. With the techniques described above, we extend $\frMG$ to $\frMGp$, which is based on the frame $\frFGp$ from the class $\CDLBWE$. By Lemma~\ref{lem:truth-lemma+}, for any formula $\phi\in\Gamma_a$, $\frMGp,\Gamma_a\Vdash\phi$. In particular, $\frMGp,\Gamma_a\Vdash\Gamma$, as required.
\end{proof}

\begin{corollary}
    Every $\Hbat$-formula valid on the frame $\frR=\langle\Real,B_<\rangle$ is a theorem of $\Bhp$.
\end{corollary}
\begin{proof}
    Suppose $\nvdash_{\Bhp}\varphi$, which means that $\{\neg\varphi\}$ is $\Bhp$-consistent. Thus by Theorem~\ref{th:completeness-for-Bhp} there is a model $\frMGp$ at which $\neg\varphi$ holds at some $\Gamma_x$. The model is based on the frame $\frFGp$ which is isomorphic to $\frR$. So $\varphi$ fails on a model based on $\frR$. In consequence, any formula valid in $\frR$ is a theorem of $\Bhp$.
\end{proof}

As is well-known, the real line $\langle\Real,<\rangle$ can be characterized up-to-isomorphism by the following conditions: it is a complete, dense linear order without endpoints which is separable, i.e., has a countable subset $D$ such every non-empty open interval has a common point with $D$. In light of mutual equivalence between strict orders and strict betweenness, this characterization can be also applied to $\frR=\langle\Real,B_{<}\rangle$. That is, the class $\SCDLBWE$ of those elements $\CDLBWE$ that meet the separability condition has $\frR$ as its only---up-to-isomorphism---element.

Let $\frI\defeq\langle [0,1],<\rangle$ be the standard binary frame on the closed interval $[0,1]$. Let us consider the product frame
\[
\frR\times\frI\defeq\langle \Real\times [0,1],<_{\ell}\rangle
\]
where $<_{\ell}$ is the lexicographic order on $\Real\times [0,1]$ obtained from $<$. As $\frR\times\frI$ is complete but not separable, the betweenness frame
\[
(\frR\times\frI)^{\ast}\defeq\langle\frR\times\frI,B_{<_{\ell}}\rangle
\]
is an element of $\CDLBWE$, which is not in $\SCDLBWE$. We can use this fact to show that separability is not expressible in the modal language extended with either satisfaction operators or existential modalities.

Given a class of frames $\Klass$ let $L(\Klass)$ be the logic of $\Klass$, i.e., the set of all formulas that are valid on every frame in $\Klass$. For a single frame $\frF$, let $L(\frF)\defeq L(\{\frF\})$. From the results obtained in this section, we know that
\[
L(\CDLBWE)=L(\frR)\,.
\]
Since $(\frR\times\frI)^{\ast}\in\CDLBWE$, it follows that $L(\frR)\subseteq L((\frR\times\frI)^{\ast})$. So, if there were a formula $\varphi$ characterizing separability of $\frR$, it would have to be in $L(\frR)$, and so in $L((\frR\times\frI)^{\ast})$, which is impossible.

\section{Summary and future work}

We have constructed two systems of hybrid logic, $\Bh$ and $\Bhp$, that are complete w.r.t. to the betweenness frame of rational numbers~$\frQ$, and the betweenness frame of the reals~$\frR$, respectively. In this way, we have shown the axiomatizability of the logic of strict betweenness in languages with $\at_i$-operators. It remains an open problem whether the logic of both frames in the basic modal similarity type (with $\Bmod$ as the only primitive operator) can be axiomatized. It seems that its solution will not be easy.

The techniques of this paper can be applied to study at least two more ternary relations. The first one is a~geometric \emph{equidistance} used by Mario Pieri \citep{Pieri-GEISNPS} to axiomatize Euclidean geometry by means of two primitives: \emph{points} and the aforementioned relation (see \citep{Gruszczynski-Pietruszczak-PS} for a modern presentation). The whole system is quite complicated (in the sense of the complexity of its axioms expressed in the primitives of the language only), but at least its fragments can be investigated from the hybrid logic perspective with a binary modal operator.

The second possibility is the hybrid approach to \emph{relative nearness} relation from \citep{vanBenthem-TLoT}. At first sight more topological than geometrical in nature, the notion is a very strong one and allows for the definition of Pieri's equidistance. The research in this direction could be then seen as a~more universal approach to ternary relations underlying geometrical structures.

\section*{Acknowledgments}
\begin{sloppypar}
Rafa\l{} Gruszczy\'{n}ski's work was funded by the National Science Center (Poland), grant number 2020/39/B/HS1/00216, ``Logico-philosophical foundations of geometry and topology''. Zhiguang Zhao's work was funded by Shandong Provincial Natural Science Foundation, China (project number: ZR2023QF021) and Taishan Young Scholars Program of the Government of Shandong Province, China (No.tsqn201909151).
\end{sloppypar}

\appendix

\section{Elementary properties of betweenness}

For completeness of the presentation in the appendix, we include the proofs of all those properties of betweenness that were used in the paper. The proofs themselves are not particularly illuminating, yet if the reader wishes to, they may check the correctness of our claims. All formulas are in open form, and what we prove are their universal closures.

Let us begin with a routine verification  that the following versions of \eqref{B4} and \eqref{B5} hold in every 3-frame satisfying \eqref{B2}, \eqref{B4}, and \eqref{B5}:
\begin{gather}
    B(x,y,z)\wedge B(y,z,u)\rarrow B(x,z,u)\,,\tag{B4$'$}\label{B4'}\\
    B(x,y,z)\wedge B(y,u,z)\rarrow B(x,u,z)\,.\tag{B5$'$}\label{B5'}
\end{gather}
\begin{proof}
\eqref{B4'} Assume that $B(x,y,z)$ and $B(y,z,u)$. Then we have $B(z,y,x)$ and $B(u,z,y)$ by \eqref{B2}. From these by \eqref{B4} we have $B(u,z,x)$, and by \eqref{B2} again we arrive at $B(x,z,u)$.

\eqref{B5'} Assume that $B(x,y,z)$ and $B(y,u,z)$. Then by \eqref{B5} we have $B(x,y,u)$. Together with $B(y,u,z)$ and \eqref{B4'} we obtain $B(x,u,z)$.
\end{proof}

\begin{lemma}\label{lem:proj-for-B}
   If $\frF\in\LBWE$, then
   \begin{align}
   \frF&{}\models B(a,x,y)\wedge B(a,x,w)\rarrow y=w\vee B(x,w,y)\vee B(x,y,w)\,,\\
   \frF&{}\models B(a,u,y)\wedge B(a,w,y)\rarrow u=w\vee B(a,u,w)\vee B(a,w,u)\,.
   \end{align}
\end{lemma}
\begin{proof}
    Ad1. Suppose ($\dagger$) $B(a,x,y)$ and ($\ddagger$) $B(a,x,w)$ and let $y\neq w$. By \eqref{B1} we obtain $\#(x,w,y)$, so by \eqref{B6} we have that
    \[
        B(x,w,y)\vee B(x,y,w)\vee  B(w,x,y)\,.
    \]
    We need to eliminate the third disjunct, so assume towards a contradiction that $B(w,x,y)$. Since by \eqref{B1} $\#(a,w,y)$, by \eqref{B6} again we obtain
    \[
        B(a,w,y)\vee B(a,y,w)\vee B(w,a,y)\,.
    \]
    In consequence, we have three cases to consider.

    (1st) $B(a,w,y)$ and $B(w,x,y)$. Then by \eqref{B5} it holds that $B(a,w,x)$, which stands in contradiction to $(\ddagger)$ by \eqref{B3}.

    \smallskip

    (2nd) $B(a,y,w)$ and $B(w,x,y)$. By \eqref{B2} it is the case that $B(y,x,w)$, and thus $B(a,y,x)$ by \eqref{B5}. Yet this contradicts $(\dagger)$ by \eqref{B3}.

    (3rd) For the last case we observe that $B(w,a,y)$ cannot hold. Indeed, suppose otherwise. From the condition and from $(\dagger)$ by means of \eqref{B5} we infer that $B(w,a,x)$. By \eqref{B2} and $(\ddagger)$ we have that $B(w,x,a)$ and so we obtain a contradiction by \eqref{B2} again.

    This concludes the proof of the first part of the lemma.

    \smallskip

    Ad 2. Assume ($\dagger$) $B(a,u,y)$ and $(\ddagger)$ $B(a,w,y)$. Let $u\neq w$. So $\#(a,u,w)$ and so
    \[
        B(a,u,w)\vee B(a,w,u)\vee B(u,a,w)
    \]
    by \eqref{B6}. Similarly as in the first part of the proof, it is sufficient to show that the third disjunct cannot obtain. To show this, we again reason by contradiction, and we assume that $B(u,a,w)$. By \eqref{B2} we have that $B(w,a,u)$, and this together with $(\dagger)$ and \eqref{B4} yield $B(w,a,y)$. We apply \eqref{B2} to this and to $(\ddagger)$ to come up with $B(y,a,w)$ and $B(y,w,a)$, which contradict \eqref{B3}.
    \end{proof}

\begin{lemma}\label{lem:exists-to-one-side}
    If $\frF\in\LBWE$, then
    \[
    \frF\models w\neq x\rarrow(\exists u\,B(x,w,u)\wedge\exists v\,B(v,x,w))\,.
    \]
\end{lemma}
\begin{proof}
    Assume $w\neq x$. By \eqref{B7} there are $a$ and $b$ such that ($\dagger$) $B(a,w,b)$. If either $a=x$ or $b=x$, we may conclude the proof. So suppose $a\neq x$ and $b\neq x$. Therefore by \eqref{B1} $\#(a,x,w)$ and $\#(b,x,w)$. So by \eqref{B6} we have:
    \[
        B(a,x,w)\qtor B(x,a,w)\qtor B(x,w,a)
    \]
    and
    \[
        B(b,x,w)\qtor B(x,b,w)\qtor B(x,w,b)\,.
    \]

    If $B(x,w,a)$ or $B(x,w,b)$ we are done.

    If $B(a,x,w)$, then from this and $(\dagger)$, applying \eqref{B2}, we have $B(w,x,a)$ and $B(b,w,a)$. Using \eqref{B5} we infer that $B(b,w,x)$, so $B(x,w,b)$ by \eqref{B2}.

    If $B(b,x,w)$, then $B(w,x,b)$ by \eqref{B2}. So from this and \eqref{B5} we deduce $B(a,w,x)$, and again we have what we need by \eqref{B2}.

    The last case is when $B(x,a,w)$ and $B(x,b,w)$. Then $B(x,a,b)$ by  the first condition, $(\dagger)$ and \eqref{B4}. So applying \eqref{B2} and \eqref{B5} we conclude that $B(a,b,w)$ which contradicts $(\dagger)$ via \eqref{B3}.

    The proof for the second conjunct is similar.
\end{proof}

\begin{lemma}\label{lem:for-D-soundness}
If $\frF\in\LBWE$, then
\[
\frF\models B(x,a,y)\wedge B(x,b,y)\wedge B(a,u,b)\rightarrow B(x,u,y)\,.
\]
\end{lemma}
\begin{proof}
Suppose $B(x,a,y)$, $B(x,b,y)$ and $B(a,u,b)$. Thus $\#(x,a,b)$, and from \eqref{B6} we obtain three possibilities
\[
B(x,a,b)\qtor B(x,b,a)\qtor B(a,x,b)\,.
\]
Similarly, $\#(y,a,b)$ and analogous three possibilities hold for this triple.

Suppose $B(a,x,b)$. From $B(x,a,y)$ and  \eqref{B2} we obtain $B(y,a,x)$, and this together with the assumption entails $B(y,x,b)$ by \eqref{B4'}. This contradicts $B(x,b,y)$ by \eqref{B2} and \eqref{B3}. Therefore we have $B(x,a,b)$ or $B(x,b,a)$. By a similar argument, we have $B(a,b,y)$ or $B(b,a,y)$. In consequence there are four cases to be considered.

(1st) $B(x,a,b)$ and $B(a,b,y)$. Using $B(a,u,b)$ we obtain $B(x,a,u)$ by \eqref{B5}, and $B(a,u,y)$ by \eqref{B5'} and \eqref{B2}. Now by \eqref{B4'} we have $B(x,u,y)$.

(2nd) $B(x,b,a)$ and $B(b,a,y)$. From $B(a,u,b)$ and \eqref{B2} we have $B(b,u,a)$, and by a similar argument to the one above we obtain $B(x,u,y)$.

(3rd) $B(x,a,b)$ and $B(b,a,y)$. From $B(x,b,y)$ and \eqref{B5} we obtain $B(x,b,a)$, which is a contradiction by \eqref{B3}.

(4th) $B(x,b,a)$ and $B(a,b,y)$. The argument is similar to the previous case, using $B(x,a,y)$.
\end{proof}

\begin{lemma}\label{lem:for-CCp->Cp}
If $\frF\in\LBWE$, then
\[
\begin{split}
    \frF\models B(y,x,z)\wedge{} &B(y_0,y,y_1)\wedge B(z_0,z,z_1)\rarrow\\
    &B(y_0,x,z_0)\vee B(y_0,x,z_1)\vee B(y_1,x,z_0)\vee B(y_1,x,z_1)\,.
\end{split}
\]
\end{lemma}

\begin{proof}
It suffices to show that one of $B(y_0,y,x)$ and $B(y_1,y,x)$ holds, and one of $B(x,z,z_0)$ and $B(x,z,z_1)$ holds. Indeed, if $B(y_i,y,x)$ and $B(x,z,z_j)$ hold (where $i,j\in\{0,1\}$), then from $B(y_i,y,x)$, $B(y,x,z)$ and \eqref{B4'} we obtain $B(y_i,x,z)$, and then from $B(x,z,z_j)$ and \eqref{B4} we have $B(y_i,x,z_j)$.

\eqref{B1} and the antecedent of the implication imply that $y\neq x$ and $y_0\neq y$. If $x=y_0$ then $B(x,y,y_1)$, and from \eqref{B2} we obtain $B(y_1,y,x)$. If $x\neq y_0$ then $\#(y_0,y,x)$, so by \eqref{B6} we have $B(y_0,y,x)$ or $B(y_0,x,y)$ or $B(y,y_0,x)$.

If $B(y_0,y,x)$ then we are done.

If $B(y_0,x,y)$, then this with $B(y_0,y,y_1)$ imply $B(y_1,y,x)$ by means of \eqref{B2} and \eqref{B5}.

If $B(y,y_0,x)$, then from this and $B(y_0,y,y_1)$ we have $B(y_1,y,x)$ by \eqref{B4} and \eqref{B2}.
\end{proof}

\bibliographystyle{plainnat}

\providecommand{\noop}[1]{}

\end{document}